\documentclass[11pt]{amsart}

\usepackage[T1]{fontenc}
\usepackage[latin1]{inputenc}
\usepackage{amsmath,amssymb,amsthm}
\usepackage{enumitem}
\usepackage{braket}
\usepackage[dvipsnames]{xcolor}
\usepackage{hyperref}

\newtheorem{theorem}{Theorem}[section]
\newtheorem{lemma}[theorem]{Lemma}
\newtheorem{proposition}[theorem]{Proposition}
\newtheorem{corollary}[theorem]{Corollary}

\theoremstyle{definition}
\newtheorem{definition}[theorem]{Definition}
\newtheorem{question}[theorem]{Question}
\newtheorem{remark}[theorem]{Remark}
\newtheorem*{remark*}{Remark}
\newtheorem{example}[theorem]{Example}

\newcommand{\Mult}{\operatorname{Mult}}
\newcommand{\Aut}{\operatorname{Aut}}
\newcommand{\dist}{\operatorname{dist}}
\newcommand{\dom}{\operatorname{dom}}

\newcommand{\N}{\mathbb{N}}

\newcommand{\B}{\mathbb{B}}

\newcommand{\mcc}{M\textsuperscript{c}Carthy}

\renewcommand{\MR}[1]{}

\title{Simply interpolating sequences in complete Pick spaces}
\author[N. Chalmoukis]{Nikolaos Chalmoukis} \address{Dipartimento di Matematica e Applicazioni, Universit\'a degli studi di Milano Bicocca, via Roberto Cozzi, 55
20125, Milano, Italy}
\email{nikolaos.chalmoukis@unimib.it}
\thanks{NC was partially supported by the Hellenic Foundation for Research and Innovation
(H.F.R.I.) under the ``2nd Call for H.F.R.I. Research Projects to support Faculty Members
\& Researchers'' (Project Number: 73342). He is also a member of GNAMPA of the Istituto Nazionale di Alta Matematica (INdAM)}
\author[A. Dayan]{Alberto Dayan}
\address{Fachrichtung Mathematik, Universit\"at des Saarlandes, 66123 Saarbr\"ucken, Germany}
\email{dayan@math.uni-sb.de}
\author[M. Hartz]{Michael Hartz}
\address{Fachrichtung Mathematik, Universit\"at des Saarlandes, 66123 Saarbr\"ucken, Germany}
\email{hartz@math.uni-sb.de}
\thanks{AD and MH were partially supported by the Emmy Noether Program of the German Research Foundation (DFG Grant 466012782)}


\begin{document}

\begin{abstract}
  We characterize simply interpolating sequences (also known as onto interpolating sequences) for complete Pick spaces.
  We show that a sequence is simply interpolating if and only if it is strongly separated.
  This answers a question of Agler and M\textsuperscript{c}Carthy.
  Moreover, we show that in many important examples of complete Pick spaces, including weighted Dirichlet spaces on the unit disc
  and the Drury--Arveson space in finitely many variables, simple interpolation does not imply multiplier interpolation.
  In fact, in those spaces, we construct simply interpolating sequences that generate infinite measures,
  and uniformly separated sequences that are not multiplier interpolating.
\end{abstract}

\maketitle

\section{Introduction}

Interpolating sequences in spaces of holomorphic functions are relevant in the context
of questions in function theory, operator theory and Banach algebras. We begin by recalling the classical setting.
Let $\mathbb{D}$ denote the open unit disc
and let $H^\infty$ be the algebra of bounded holomorphic functions on $\mathbb{D}$.
Given a sequence $(z_n)$ in $\mathbb{D}$, one can consider the restriction operator
\begin{equation*}
  R: H^\infty \to \ell^\infty, \quad f \mapsto (f(z_n))_n.
\end{equation*}
The sequence $(z_n)$ is said to be interpolating for $H^\infty$ if $R$ is surjective.
Such sequences were characterized by Carleson \cite{Carleson58}; we will recall the characterization below.

Defining interpolating sequences for the Hardy space $H^2$ is slightly more subtle.
Functions in $H^2$ need not be bounded; instead, they obey the growth condition $(1 - |z|^2) |f(z)|^2  \le \|f\|^2$
for all $z \in \mathbb{D}$. Thus, it is natural to consider the weighted restriction operator
\begin{equation*}
  T: f \mapsto ( (1 - |z_n|^2)^{1/2} f(z_n))_n.
\end{equation*}
In contrast with the setting for $H^\infty$, it is not automatic that $T$ maps $H^2$ into $\ell^2$.
Thus, there are two possible notions of interpolating sequences for $H^2$, given by requiring that $T H^2 = \ell^2$,
or that $T H^2 \supset \ell^2$.
Shapiro and Shields \cite{SS61} showed that the two notions actually coincide, and are in turn
equivalent to being an interpolating sequence for $H^\infty$.

In this article, we study interpolating sequences for normalized complete Pick spaces.
Precise definitions will be given later; for now, we only remark that important examples of normalized complete Pick spaces are the Hardy space $H^2$
and the classical Dirichlet space $\mathcal{D}$ on the unit disc. Moreover, for each $a \in (0,1)$, the standard weighted Dirichlet space $\mathcal{D}_a$,
which is the reproducing kernel Hilbert space on $\mathbb{D}$ with kernel $k_w(z) = \frac{1}{(1 - z \overline{w})^a}$,
is a normalized complete Pick space; see \cite[Section 7]{AM02} or \cite[Section 4.2]{Hartz22a}.

Another frequently studied normalized complete Pick space is the Drury--Arveson space $H^2_d$, which is the reproducing
kernel Hilbert space on the Euclidean unit ball $\mathbb{B}_d$ in $\mathbb{C}^d$ with kernel $k_w(z) = \frac{1}{1 - \langle z,w \rangle }$.
Among other things, this space plays a key role in multivariable operator theory; see \cite{Arveson98,Hartz22a,Shalit13} for background.

Let $\mathcal{H}$ be a reproducing kernel Hilbert space of functions on a set $X$.
Let $k: X \times X \to \mathbb{C}$ be the reproducing kernel of $\mathcal{H}$. Thus,
\begin{equation*}
  \langle f, k(\cdot,w) \rangle = f(w) \quad \text{ for all } w \in X, f \in \mathcal{H}.
\end{equation*}
We will assume throughout that $k(z,z) \neq 0$ for all $z \in X$.
The Cauchy--Schwarz inequality shows that $|f(z)| \le \|f\| ~k(z,z)^{1/2}$ for all $f \in \mathcal{H}$
and $z \in X$. Given a sequence $(z_n)$ in $X$, we consider the weighted restriction operator
\begin{equation*}
  T: f \mapsto \Big( \frac{f(z_n)}{k(z_n,z_n)^{1/2}} \Big)_n.
\end{equation*}
A sequence $(z_n)$ in $X$ is said to be simply interpolating (SI) (or onto interpolating) if
\begin{equation*}
  T(\mathcal{H}) \supset \ell^2.
\end{equation*}

Simply interpolating sequences have been studied in \cite{Arcozzi16, Bishop94, Chalmoukis21, Chalmoukis23, MS94a}, see also \cite[Chapter 8]{ArcRochSawWick19}. Simply interpolating sequences are closely related with interpolationg sequences
for the multiplier algebra
as well as with separation and Carleson measure conditions.
We briefly recall the relevant notions.
The multiplier algebra of $\mathcal{H}$ is
\begin{equation*}
  \Mult(\mathcal{H}) = \{ \varphi: X \to \mathbb{C}: \varphi \cdot f \in \mathcal{H} \text{ for all } f \in \mathcal{H}\},
\end{equation*}
equipped with the multiplier norm $\|\varphi\|_{\Mult(\mathcal{H})} = \sup \{ \|\varphi \cdot f\|: \|f\|_{\mathcal{H}} \le 1 \}$.
Every multiplier is bounded, so we may consider the operator
\begin{equation*}
  R: \Mult(\mathcal{H}) \to \ell^\infty, \quad \varphi \mapsto (\varphi(z_n))_n.
\end{equation*}
Given $z,w \in X$, we let
\begin{equation*}
  d(z,w) = \sqrt{ 1 - \frac{|k(z,w)|^2}{k(z,z) k(w,w)}}.
\end{equation*}
Then $d$ is a pseudo-metric on $X$; see \cite[Lemma 9.9]{AM02} and also \cite{ARS+11} for background.
If $\mathcal{H} = H^2$, this is the pseudo-hyperbolic metric on the disc.

A sequence $(z_n)$ is said to
\begin{itemize}
  \item[(IS)] be an interpolating sequence (for $\Mult(\mathcal{H})$) if $R(\Mult(\mathcal{H})) = \ell^\infty$;
  \item[(C)] satisfy the Carleson measure condition if $T$ maps $\mathcal{H}$ boundedly into $\ell^2$, i.e.\ if there exists $C \ge 0$ with
    \begin{equation*}
      \sum_{n} \frac{|f(z_n)|^2}{k(z_n,z_n)} \le C \|f\|^2 \quad \text{ for all } f \in \mathcal{H};
    \end{equation*}
  \item[(US)] be uniformly separated if $\inf_{n} \prod_{j \neq n} d(z_j,z_n) > 0$.
  \item[(SS)] be strongly separated if there exists $\varepsilon > 0$ so that for each $n \in \mathbb{N}$,
    there exists $\varphi \in \Mult(\mathcal{H})$ of multiplier norm at most one with
    $\varphi(z_n) = \varepsilon$ and $\varphi(z_j) = 0$ for $j \neq n$.
  \item[(WS)] weakly separated if there exists $\varepsilon > 0$ such that $d(z_n,z_j) \ge \varepsilon$
    if $n \neq j$.
\end{itemize}
By the closed graph theorem, the Carleson measure condition is also equivalent to demanding that $T( \mathcal{H}) \subset \ell^2$.

The best understood complete Pick space in the Hardy space $H^2$. In that case,
the previously mentioned theorems of Carleson \cite{Carleson58} and Shapiro--Shields \cite{SS61} combine to show that
\begin{equation*}
  \text{(IS)} \Leftrightarrow \text{(C)} + \text{(SI)} \Leftrightarrow \text{(C)} + \text{(WS)}
  \Leftrightarrow \text{(US)} \Leftrightarrow \text{(SS)} \Leftrightarrow \text{(SI)}.
\end{equation*}
For a modern treatment; see \cite[Chapter 9]{AM02} or \cite[Chapter VII]{Garnett07}.

In the setting of all complete Pick spaces, the following implications are known:
\begin{equation*}
  \text{(IS)} \Leftrightarrow \text{(C)} + \text{(SI)} \Leftrightarrow \text{(C)} + \text{(WS)}
  \Rightarrow \text{(US)} \Rightarrow \text{(SS)} \Leftarrow \text{(SI)}.
\end{equation*}
The first equivalence is a theorem of Marshall and Sundberg \cite{MS94a}, see also \cite[Theorem 9.19]{AM02}. The second equivalence
was established by Aleman, \mcc, Richter and the third named author \cite{AHM+17}.
The implication (C) + (WS) $\Rightarrow$ (US) is due to Agler and \mcc\ \cite[Theorem 9.43]{AM02}.
The remaining implications can also be found in Chapter 9 of \cite{AM02};
see also Section \ref{sec:prelim} below for more explanation.

The goal of this article is to determine to what extent the remaining implications hold in general complete Pick spaces.
Besides the Hardy space, this is also known in the case of the Dirichlet space.
In that case, Bishop established the equivalence (SS) $\Leftrightarrow$ (SI) \cite[Theorem 1.2]{Bishop94} (see \cite{Chalmoukis21} for a geometric description of onto interpolating sequences).

Our first main result shows that the equivalence (SS) $\Leftrightarrow$ (SI) holds in any complete Pick space.
This answers a question raised by Agler and \mcc\ \cite[p.\ 145]{AM02}.

\begin{theorem}
  \label{thm:simply_interpolating_intro}
  Let $\mathcal{H}$ be a normalized complete Pick space of functions on a set $X$. A sequence in $X$ is simply interpolating if and only
  if it is strongly separated.
\end{theorem}

As mentioned above, this result was known in the Hardy space and in the Dirichlet space.
However, our proof is entirely different from the known proofs in those spaces.
Indeed, the proofs in the Hardy space (see e.g.\ \cite{Carleson58} and \cite[Chapter 9]{AM02}) make use of Blaschke products, while Bishop's proof in the Dirichlet space heavily relies on potential theoretic arguments. Both of these techniques are not available in the generality of all complete Pick spaces.

Instead, we use the positive solution of the Kadison--Singer problem, or more precisely of the Feichtinger conjecture, due to Marcus, Spielman and Srivastava. This theorem was also used in the proof of (WS) + (C) $\Rightarrow$ (IS) in \cite{AHM+17}.
However, as we will explain later, in the absence of the Carleson measure condition,
the Marcus--Spielman--Srivastave theorem does not apply directly.
We overcome this by showing in Section \ref{sec:um_feichtinger}, as a consequence of the Marcus--Spielman--Srivasta theorem, that any uniformly
minimal sequence of unit vectors in a Hilbert space is a finite union of sequences that satisfy a lower
Riesz condition. In the setting of interpolating sequences, this means that any strongly separated sequence is a finite union
of simply interpolating sequences; this is even true without the complete Pick property.
In a separate step, we use the complete Pick property to pass from a finite union of simply interpolating sequences
to a single simply interpolating sequence; this is done in Section \ref{sec:si_ss}.

It is worth remarking that Theorem \ref{thm:simply_interpolating_intro}, or more precisely the implication (SS) $\Rightarrow$ (SI),
is in some sense a refinenemt of the implication (WS) + (C) $\Rightarrow$ (IS) established in \cite{AHM+17}.
Indeed, from (SS) $\Rightarrow$ (SI) and from the previously known results (WS) + (C) $\Rightarrow$ (SS) and (SI) + (C) $\Rightarrow$ (IS),
we immediately obtain (WS) + (C) $\Rightarrow$ (IS).
Thus, it is not an accident that our proof of (SS) $\Rightarrow$ (SI) also uses the Marcus--Spielman--Srivastava theorem.
On the other hand, there is now a second proof of the implication (WS) + (C) $\Rightarrow$ (IS) that does not rely on the Marcus--Spielman--Srivastava
theorem. This proof uses the column-row property of complete Pick spaces; see \cite{Hartz20}.
The question of whether the column-row property can also be used to prove (SS) $\Rightarrow$ (SI) remains open.

Having established the equivalence (SS) $\Leftrightarrow$ (SI), the natural remaining questions
are in which spaces the implications (SS) $\Rightarrow$ (US) and (US) $\Rightarrow$ (IS) hold.
As mentioned above, it is known that both of these implications fail in the Dirichlet space.
A fairly simple example shows that (SS) $\not \Rightarrow$ (US) in $H^2_\infty$,
the Drury--Arveson space on the infinite dimensional ball; see \cite[Example 9.55]{AM02}.
Our second main result shows that both implications also fail in the Drury-Arveson space in dimension at least two,
and in standard weighted Dirichlet spaces.
In this context, another relevant condition is the following: we say that a sequence $(z_n)$ generates a finite measure (FM)
if
\begin{equation*}
  \sum_{n} \frac{1}{k(z_n,z_n)} < \infty.
\end{equation*}
Note that in $H^2$, these are precisely the Blaschke sequences.
It is not hard to see that (US) $\Rightarrow$ (FM),
see Proposition \ref{prop:US_FM} below.
Our various counterexamples can then be summarized as follows.

\begin{theorem}
  \label{thm:counterexamples_intro}
  Let $\mathcal{H}$ be any of the spaces $H^2_d$ for $d \ge 2$, or $\mathcal{D}_a$ for $a \in (0,1)$.
  Then:
  \begin{enumerate}[label=\normalfont{(\alph*)}]
    \item There exists a sequence $(z_n)$ that is uniformly separated, but not interpolating.
    \item There exists a sequence $(z_n)$ that is strongly separated and that generates a finite measure, but that is not uniformly separated.
    \item There exists a sequence $(z_n)$ that is strongly separated, but that does not generate a finite measure, and hence is
      not uniformly separated.
  \end{enumerate}
\end{theorem}

In summary, the valid implications in any complete Pick space are
\begin{equation*}
  \text{(IS)} \Leftrightarrow \text{(C)} + \text{(SI)} \Leftrightarrow \text{(C)} + \text{(WS)}
  \Rightarrow \text{(US)} \Rightarrow \text{(SS) + (FM)} \Rightarrow \text{(SS)} \Leftrightarrow \text{(SI)}.
\end{equation*}
Moreover, none of the three implications are reversible in $H^2_d$ for $d \ge 2$ or in $\mathcal{D}_\alpha$ for $\alpha \in (0,1)$.
One should note that, for the Dirichlet space $\mathcal{D}$, the failure of these implications is known. Marshall and Sundberg constructed in \cite{MS94a} a sequence that is uniformly separated and that does not satisfy the Carleson measure condition. Arcozzi, Rochberg and Sawyer constructed in \cite{Arcozzi16} the remaining two counterexamples for $\mathcal{D}$, see also \cite[Section 8.4.4]{ArcRochSawWick19}. (The second example in \cite[Section 8.4.4]{ArcRochSawWick19} is a finite measure simply interpolating sequence that does not satisfy the ``simple condition'' discussed in \cite{ArcRochSawWick19}; the construction yields that it is not uniformly separated.) Another example of a finite measure simply interpolating sequence that is not uniformly separated is the final example in \cite{Chalmoukis23}. \\

The remainder of the paper is organized as follows. Section \ref{sec:prelim} contains a brief overview of the background results on interpolating sequences and complete Pick kernels which are necessary for the proof of our main results.\\
Section \ref{sec:um_feichtinger} uses an argument based on the positive solution of the Feichtinger conjecture to show that any uniformly minimal sequence of unit vectors in a Hilbert space is a finite union of sequences that satisfy a lower Riesz condition. As a consequence, any strongly separated sequence is the union of finitely many simply interpolating sequences, independently of the complete Pick property.\\
Section \ref{sec:si_ss} will then show that strong separation implies simple interpolation for the whole sequence, provided that the underlying reproducing kernel Hilbert space enjoys the complete Pick property. This proves Theorem \ref{thm:simply_interpolating_intro}.\\
Section \ref{sec:us_not_is} is devoted to the construction of uniformly separated sequences that are not interpolating in $\mathcal{D}_a$, $0<a<1$, proving Theorem \ref{thm:counterexamples_intro} (a) for such spaces. \\
In Section \ref{sec:si_im} we prove Theorem \ref{thm:counterexamples_intro} (c) for the spaces $\mathcal{D}_a$, $0<a<1$, by invoking general results on separation for orbits of discrete subgroups of automorphisms of the unit disc. For the range $0<a\leq1/2$, we use some of the tools from Section \ref{sec:us_not_is} to give a proof of Theorem \ref{thm:counterexamples_intro} (b), see Example \ref{example:circulant}. \\
Finally, in Section \ref{sec:ball}, we lift the counterexamples considered in Section \ref{sec:us_not_is} and Section \ref{sec:si_im} from $\mathcal{D}_\frac{1}{2}$ to the Drury-Arveson space on the $2$-dimensional unit ball, and prove Theorem \ref{thm:counterexamples_intro} (b).

\section{Preliminaries}
\label{sec:prelim}

\subsection{Complete Pick spaces}

For a detailed account on complete Pick spaces, the reader is referred to the book
\cite{AM02}. A brief introduction can also be found in \cite[Section 4]{Hartz22a}. We recall the main points that we need.

Let $\mathcal{H}$ be a reproducing kernel Hilbert space of functions on a set $X$,
with reproducing kernel $k$ and multiplier algebra $\Mult(\mathcal{H})$.
We will always assume that $k(z,z) \neq 0$ for all $z \in X$.
The space $\mathcal{H}$ is said to be a Pick space if, for all $n \in \mathbb{N}$,
all $z_1,\ldots,z_n \in $ and $\lambda_1,\ldots,\lambda_n \in \mathbb{C}$
for which the Pick matrix
\begin{equation*}
  \big[ (1 - \lambda_j \overline{\lambda_i}) k(z_j,z_i) \big]_{j,i=1}^n
\end{equation*}
is positive, there exists a multiplier $\varphi \in \Mult(\mathcal{H})$ with
\begin{equation*}
  \varphi(z_j) = \lambda_j \quad (j= 1,2,\ldots, n)
\end{equation*}
and $\|\varphi\|_{\Mult(\mathcal{H})} \le 1$.
If the analogous condition holds for matrix valued interpolation (with $\overline{\lambda_j}$ replaced with
the adjoint  $\lambda_j^*$), then $\mathcal{H}$ is said to be a complete Pick space.
A routine compactness argument shows that if $\mathcal{H}$ is a (complete) Pick space,
then one can solve (matrix or operator valued) interpolation problems with infinitely many nodes,
provided all Pick matrices corresponding to finitely many nodes are positive.

The kernel $k$ is said to be normalized at $z_0 \in X$ if $k(z,z_0) = 1$ for all $z \in X$.
A complete Pick space with a normalized kernel is also called a normalized complete Pick space.
According to results of Quiggin \cite{Quiggin93}, McCullough \cite{McCullough92}, Agler and \mcc\ \cite{AM00}, a reproducing kernel Hilbert
space with kernel normalized at $z_0 \in X$
is a complete Pick space if and only if there exists a map $b$ from $X$ into the unit ball of a Hilbert space
such that $b(z_0) = 0$ and such that
\begin{equation}
  \label{eqn:CNP}
  k(z,w) = \frac{1}{1 - \langle b(z), b(w) \rangle} \quad (z,w \in X);
\end{equation}
see for instance \cite[Theorem 7.31]{AM02}.
The prototypical example of a complete Pick space is the Hardy space $H^2$,
but there are many other examples, including
the classical Dirichlet space $\mathcal{D}$ \cite{Agler88a}, the standard weighted Dirichlet space $\mathcal{D}_a$
for $0 < a < 1$, and the Drury--Arveson space $H^2_d$; see e.g.\ \cite[Chapter 7]{AM02} and \cite[Section 4]{Hartz22a}.

\subsection{Gramians}

Let $\mathcal{H}$ be a reproducing kernel Hilbert space on $X$, with kernel $k$.
Given a set $Z = \{ z_n: n \in \mathbb{N}\}$ of points in $X$,
let $f_n = \frac{k(\cdot,z_n)}{\|k(\cdot,z_n)\|}$ be the normalized reproducing kernel
at $z_n$ and let
\begin{equation*}
  G = \big[ \langle f_i,f_j \rangle \big]_{i,j} =  \Big[ \frac{k(z_i,z_j)}{ k(z_i,z_i)^{1/2} k(z_j,z_j)^{1/2}} \Big]_{i,j}
\end{equation*}
be the Gramian.
It is now well known that the various interpolation, Carleson measure and separation conditions
discussed in the introduction can be rephrased as conditions about the Gramian; see \cite[Chapter 9]{AM02}.

We recall the facts we need and provide some comments about proofs for the convenience of the reader.
Throughout, let $\mathcal{H}$ be a reproducing kernel Hilbert space on $X$, let $(z_n)$ be a sequence
  in $X$ and let $G$ be the Gramian.
We also let
\begin{equation*}
  T: f \mapsto \Big( \frac{f(z_n)}{K(z_n,z_n)^{1/2}} \Big)_n.
\end{equation*}
Recall that by definition, $(z_n)$ satisfies the Carleson measure condition if $T$ is a bounded operator
from $\mathcal{H}$ into $\ell^2$.
The following result was already used by Shapiro and Shields in the context of interpolating sequences;
see \cite[Theorem 3]{SS61}.

\begin{proposition}
  \label{prop:BG}
  The sequence $(z_n)$ satisfies the Carleson measure condition if and only if $G$ is bounded
  as an operator on $\ell^2$. In this case,
  \begin{equation*}
    \|G\| =
  \sup \Big\{ \sum_{n} \frac{|f(z_n)|^2}{k(z_n,z_n)}: \|f\|_{\mathcal{H}} \le 1 \Big\}.
  \end{equation*}
\end{proposition}

\begin{proof}
  This is proved \cite[Theorem 3]{SS61} and in \cite[Proposition 9.5]{AM02}. To see the formula for the norm,
  note that the adjoint $T^*: \ell^2 \to \mathcal{H}$ maps $e_n$ to $f_n$, and so $G = T T^*$.
  Thus,  $\|G\| = \|T\|^2$, which in turn equals the supremum in the statement.
\end{proof}

Next, we phrase the uniform separation condition in terms of the Gramian.
This is essentially contained in \cite[Theorem 9.43]{AM02}.
We call $\inf_n \prod_{j \neq n} d(z_j,z_n)$ the \emph{uniform separation constant} of the sequence.
Similarly, $\inf_{n \neq j} d(z_n,z_j)$ is called the \emph{weak separation constant} of the sequence.
Finally, note that $G$ defines a bounded operator from $\ell^2$ to $\ell^\infty$ if and only if
  \begin{equation*}
    \sup_{n} \sum_{j} \frac{|k(z_n,z_j)|^2}{k(z_n,z_n) k(z_j,z_j)} < \infty.
  \end{equation*}

\begin{proposition}
  \label{prop:uniformly_separated}
  The sequence $(z_n)$ is uniformly separated if and only if it is weakly separated
  and $G$ defines a bounded operator from $\ell^2$ to $\ell^\infty$.
  Moreover, the uniform separation constant can be bounded below only in terms
  of the weak separation constant and $\|G\|_{\ell^2 \to \ell^\infty}$.
\end{proposition}

\begin{proof}
  Clearly, weak separation is necessary for uniform separation. On the other hand, if $(z_n)$ is weakly separated then the individual terms in the infinite product
  \begin{equation*}
    \prod_{j \neq n} d(z_j,z_n)
  \end{equation*}
  are uniformly bounded below, hence the logarithm of the infinite product is comparable
  to
  \begin{equation*}
    - \sum_{j \neq n} (1 - d(z_j,z_n)^2) = - \sum_{j} \frac{|k(z_n,z_j)|^2}{k(z_n,z_n) k(z_j,z_j)},
  \end{equation*}
  from which all claims follow.
\end{proof}

Note that the implication (WS) + (C) $\Rightarrow$ (US), which was mentioned
in the introduction, is immediate from Propositions \ref{prop:BG} and \ref{prop:uniformly_separated}.
For Pick spaces, the implication (US) $\Rightarrow$ (SS) can be seen by taking
an infinite product of solutions of two point Pick problems; see \cite[Theorem 9.43]{AM02}.

The fact that (US) $\Rightarrow$ (FM) for complete Pick spaces follows from the following simple fact.

\begin{proposition}
  \label{prop:US_FM}
  Let $\mathcal{H}$ be a normalized complete Pick space.
  Suppose that one column of $G$ belongs to $\ell^2$. Then $(z_n)$ generates a finite measure.
\end{proposition}

\begin{proof}
  The particular form of a normalized complete Pick kernel $k$ given by Equation \eqref{eqn:CNP}
  shows that $|k(z,w)| \ge \frac{1}{2}$. So if the $n$-th column of $G$ belongs to $\ell^2$, then
  \begin{equation*}
    \sum_{j} \frac{|k(z_n,z_j)|^2}{k(z_n,z_n) k(z_j,z_j)} 
    \ge \frac{1}{4 k(z_n,z_n)} \sum_{j} \frac{1}{k(z_j,z_j)}. \qedhere
  \end{equation*}
\end{proof}

Next, we reformulate the simple interpolation condition.
We say that $G = (G_{ij})$ is bounded below if there exists $\varepsilon > 0$
such that $(G_{ij})_{i,j=1}^N \ge \varepsilon I_N$ for all $N \in \mathbb{N}$.
Recall further that a sequence $(f_n)$ of unit vectors
satisfies a lower Riesz condition if there exists $\varepsilon > 0$ such that
\begin{equation*}
  \Big\| \sum_{n} \alpha_n f_n \Big\|^2 \ge \varepsilon \sum_{n} |\alpha_n|^2
\end{equation*}
for all finitely supported sequences $(\alpha_n)$.
We also use $\bigvee$ to denote closed linear spans in Hilbert spaces.

The following result was also used by Shapiro and Shields \cite[Theorem 3]{SS61},
see also \cite[Proposition 9.18]{AM02}.
We sketch a slight variation of the proofs presented there.

\begin{proposition}
  \label{prop:SI_Gramian}
  The following are equivalent:
  \begin{enumerate}[label=\normalfont{(\roman*)}]
    \item $(z_n)$ is a simply interpolating sequence;
    \item the normalized reproducing kernels at the points $(z_n)$ satisfy a lower Riesz condition;
    \item $G$ is bounded below.
  \end{enumerate}
\end{proposition}

\begin{proof}
  The equivalence of (ii) and (iii) is immediate from the definitions.

  (i) $\Rightarrow$ (ii)
  Let $f_n$ be the normalized reproducing kernel at $z_n$.
  The weighted restriction operator $T$ introduced at beginning of the subsection can
  also be described by
  \begin{equation*}
    T: f \mapsto (\langle f,f_n \rangle).
  \end{equation*}
  Let $\dom(T) = \{f \in \mathcal{H}: T f \in \ell^2\}$. Note that $\ker(T) = \{f \in \mathcal{H}: f(z_n) = 0 \text{ for all } n \in \mathbb{N}\}$, and $\ker(T)^{\bot} = \bigvee_n f_n$.
  Now, if $(z_n)$ is simply interpolating, then
  \begin{equation*}
    T: \dom(T) \cap \ker(T)^\bot \to \ell^2
  \end{equation*}
  is a closed bijective operator. Endowing $\dom(T) \cap \ker(T)^\bot$ with the graph norm for a moment,
  the open mapping theorem yields a bounded inverse $T^{-1}: \ell^2 \to \ker(T)^\bot$.
  Then $T^{-1}$ is in particular bounded with respect to the norm of $\mathcal{H}$.
  A small computation shows that $(T^{-1})^*$ maps $f_n$ to $e_n$.
  In particular, $(f_n)$ satisfies a lower Riesz condition.

  (i) $\Rightarrow$ (ii) If the normalized reproducing kernels $f_n$ satisfy a lower Riesz
  condition, then there exists a bounded linear operator
  \begin{equation*}
    S: \mathcal{H} \to \ell^2, \quad f_n \mapsto e_n.
  \end{equation*}
  Another small compution shows that $T S^* (\alpha_n) = (\alpha_n)$
  for all $(\alpha_n) \in \ell^2$. Hence $(z_n)$ is simply interpolating.
\end{proof}

Finally, we reformulate the strong separation condition.
We first require the following definition.

\begin{definition}
  \label{defn:1}
  Let $(g_n)$ be a sequence of vectors in a Hilbert space $\mathcal{H}$.
  We say that $(g_n)$ is \emph{$\varepsilon$-uniformly minimal} if
  \begin{equation*}
    \dist\left(g_n, \bigvee_{j \neq n} g_j\right) \ge \varepsilon.
  \end{equation*}
  We say that $(g_n)$ is uniformly minimal if it is $\varepsilon$-uniformly minimal for some $\varepsilon > 0$.
\end{definition}

The following lemma is well known.

\begin{lemma}
  \label{lem:dual_system}
  Let $(g_n)$ be sequence of vectors in a Hilbert space. The following are equivalent:
  \begin{enumerate}[label=\normalfont{(\roman*)}]
    \item $(g_n)$ is $\varepsilon$-uniformly minimal;
    \item there exists a sequence $(h_n)$ of vectors in $\bigvee_n g_n$ with $\langle g_n, h_j  \rangle = \delta_{n j}$
      and $\|h_n\| \le \frac{1}{\varepsilon}$ for all $n,j \in \mathbb{N}$.
  \end{enumerate}
\end{lemma}

\begin{proof}
  (ii) $\Rightarrow$ (i) The Cauchy--Schwarz inequality shows that for all $x \in \bigvee_{j \neq n} g_j$, we have
  \begin{equation*}
    \|h_n\| \|g_n - x\| \ge | \langle g_n -x, h_n \rangle| =1.
  \end{equation*}
  So $\|g_n - x\| \ge \varepsilon$.

  (i) $\Rightarrow$ (ii) Let $x_n$ be the orthogonal projection of $g_n$ into $\bigvee_{j \neq n} g_j$
  and let $h_n = \frac{g_n - x_n}{\|g_n - x_n\|^2}$. Then $\langle g_n, h_j \rangle  = \delta_{nj}$
  and 
  \begin{equation*}
    \|h_n\| = \frac{1}{\dist(g_n, \bigvee_{j \neq n} g_j)} \le \frac{1}{\varepsilon}. \qedhere
  \end{equation*}
\end{proof}

The sequence $(h_n)$ in (ii) is unique and is called the minimal dual system of $(g_n)$.

A sequence $(z_n)$ in $X$ is called \emph{$\varepsilon$-strongly separated} if
for each $n$, there exists a multiplier $\varphi_n$ of multiplier norm at most one with $\varphi_n(z_j) = \varepsilon \delta_{nj}$.
Thus, $(z_n)$ is strongly separated if and only if it is $\varepsilon$-strongly separated for some $\varepsilon > 0$.

The following standard result gives the connection between uniform minimality and strong separation.

\begin{proposition}
  \label{prop:SS_Pick}
  Let $f_n = \frac{k(\cdot,z_n)}{k(z_n,z_n)^{1/2}}$ be the sequence
  of normalized reproducing kernels at the points $(z_n)$. Let $\varepsilon > 0$.
  Consider the following statements:
  \begin{enumerate}[label=\normalfont{(\roman*)}]
    \item $(f_n)$ is $\varepsilon$-uniformly minimal.
    \item $(z_n)$ is $\varepsilon$-strongly separated.
  \end{enumerate}
  Then (ii) $\Rightarrow$ (i). If $\mathcal{H}$ is a Pick space, then (i) $\Leftrightarrow$ (ii).
\end{proposition}

\begin{proof}
  (ii) $\Rightarrow$ (i)
  Let $\varphi_n$ be as the definition of $\varepsilon$-strong separation.
  Let $h_n$ be the orthogonal projection of $\varphi_n f_n$
  onto $\bigvee_{j} f_{j}$. Then $\|h_n\| \le 1$.
  If $j \neq n$, then
  \begin{equation*}
    \langle h_n, f_j \rangle = \langle \varphi_n f_n, f_j \rangle  = 0,
  \end{equation*}
  as $\varphi_n(z_j) = 0$. For $j=n$, we have
  \begin{equation*}
    \langle h_n, f_n \rangle = \langle \varphi_n f_n, f_n \rangle = \varepsilon.
  \end{equation*}
  From this, it follows that $\dist(f_n, \bigvee_{j \neq n} f_j) \ge \varepsilon$, so $(f_n)$ is $\varepsilon$-uniformly minimal; see Lemma \ref{lem:dual_system}.

  (i) $\Rightarrow$ (ii) Assume now that $\mathcal{H}$ is a Pick space and let $(f_n)$ be $\varepsilon$-strongly
  separated.
  Let $(h_n)$ be the minimal dual system. Then $\|h_n\| \le \frac{1}{\varepsilon}$ by Lemma \ref{lem:dual_system}.
  Thus, for each $n$, the map
  \begin{equation*}
    T_n: \bigvee_{j} f_j \to \bigvee_j f_j, \quad f \mapsto \varepsilon \langle f, h_n \rangle f_n,
  \end{equation*}
  has norm at most $1$, and $T_n(f_j) = 0$ if $j \neq n$ and $T_n(f_n) = \varepsilon f_n$.
  By the Pick property, $T_n$ extends to the adjoint of a multiplication operator $M_{\varphi_n}$,
  with $\|\varphi_n\|_{\Mult(\mathcal{H})} \le 1$, and $\varphi_n(z_j) = \varepsilon \delta_{nj}$.
\end{proof}

It is clear that any sequence that satisfies a lower Riesz condition is uniformly minimal.
Thus, Propositions \ref{prop:SI_Gramian} and \ref{prop:SS_Pick} show that (SI) $\Rightarrow$ (SS) in complete Pick spaces.
Proving Theorem \ref{thm:simply_interpolating_intro} amounts to showing that the converse
holds.

We will also use the following multiplier algebra formulation of simply interpolating sequences.
It is contained in the proof of \cite[Theorem 9.46]{AM02}.

\begin{proposition}
  \label{prop:row}
  Let $\mathcal{H}$ be a normalized complete Pick space.
  Then $G \ge \varepsilon^2 I$ if and only if there exists a row multiplier $\Psi \in \Mult(\mathcal{H} \otimes \ell^2,\mathcal{H})$ of norm at most one
  with $\Psi(z_n) = \varepsilon e_n^T$ for all $n$.
\end{proposition}

\begin{proof}
  See \cite[Theorem 9.46]{AM02}. This also follows directly from the complete Pick property. Indeed, such a multiplier exists if and only if the matrix 
  \[ \big[ k(z_n,z_m)(1-\varepsilon^2 \langle e_n , e_m \rangle_{\ell^2(\mathbb{N})} )  \big]_{n,m} = \big[ k(z_n,z_m) (1 - \varepsilon^2 \delta_{nm})\big]_{n,m}\]
  is positive semi-definite. Equivalently, by rescaling, if and only if
  \[ \bigg[ \frac{ k(z_n,z_m) }{\sqrt{k(z_n,z_n) k(z_m,z_m)}}(1 - \varepsilon^2 \delta_{nm}) \bigg]_{n,m} = G - \varepsilon^2 I \geq 0. \qedhere \]
\end{proof}

The implication (SI) $\Rightarrow$ (SS) can also be seen from this result,
since the entries of the row multiplier $\Psi$ show strong separation of the sequence $(z_n)$.

\section{Uniform minimality and the Feichtinger conjecture}
\label{sec:um_feichtinger}

In this section, we consider Gramians of general unit vectors in Hilbert spaces.
We will later apply the results to normalized kernel vectors in complete Pick spaces,
but the arguments in this section are applicable in greater generality.

We begin by recalling the positive solution to the Feichtinger conjecture  due to Marcus, Spielman and Srivastava \cite{MSS15}.
Their theorem says that any sequence of unit vectors in a Hilbert space whose Gramian is bounded is a finite union
of sequences whose Gramians are bounded and bounded below.
We will use a quantitative refinement due to Bownik, Casazza, Marcus and Speegle \cite{BCM+15}.

\begin{theorem}[Bownik--Casazza--Marcus--Speegle]
  \label{thm:Feichtinger_Bessel}
  Let $(f_n)$ be a sequence of vectors in a Hilbert space with $\|f_n\|^2 \ge \varepsilon > 0$.
  Assume that the Gramian of $(f_n)$ has norm at most $1$. Then $(f_n)$ can be partitioned into $O(\frac{1}{\varepsilon})$ sequences,
  each of whose Gramian is bounded below by $\varepsilon_0 \varepsilon$.
  Here, $\varepsilon_0$ and the constant implicit in $O(\frac{1}{\varepsilon})$ are universal constants.
\end{theorem}

The statement for finite sequences (which is all we need) is \cite[Theorem 6.7]{BCM+15}; see \cite[Theorem 6.11]{BCM+15}
for the extension to infinite sequences.

Unlike in the proof of (WS) + (C) $\Rightarrow$ (IS) in \cite{AHM+17},
Theorem \ref{thm:Feichtinger_Bessel} does not apply directly to our situtation.
This is because the Gramian of a sequence of normalized reproducing kernels is bounded if and only if the Carleson measure condition holds; see Proposition \ref{prop:BG}.
But the Carleson measure condition need not hold for simply interpolating sequences; indeed, this is what distinguishes
simply interpolating sequences from interpolating sequences for the multiplier algebra.

Proposition \ref{prop:uniformly_separated} shows that we need to work with uniformly minimal sequences.
With the help of Theorem \ref{thm:Feichtinger_Bessel}, we will establish the following result.

\begin{theorem}
  \label{thm:Feichtinger}
  Let $(f_n)$ be a sequence of vectors in a Hilbert space that is uniformly minimal.
  Then $(f_n)$ is a finite union of sequences whose Gramian is bounded below.

  More precisely, there exist universal constants $\tau,M > 0$ so that
if $(f_n)$ is $\varepsilon$-uniformly minimal with $0 <\varepsilon \le 1$, then $(f_n)$ can be partitioned into $\lfloor \frac{M}{\varepsilon^2} \rfloor$ sequences,  each of whose Gramian is bounded below by $\tau \varepsilon^2$.
\end{theorem}

The proof of Theorem \ref{thm:Feichtinger} occupies the remainder of this section.

Recall that if $(f_n)_{n=1}^N$ is a finite sequence of linearly independent vectors in a Hilbert space,
then the minimal dual system of $(f_n)$ is the unique sequence $(g_n)_{n=1}^N$ in $\bigvee_{n=1}^N f_n$
satisfying $\langle f_n,g_j \rangle = \delta_{nj}$ for all $1 \le n,j \le N$.

We require two elementary linear algebra lemmas, which are well known.
For the convenience of the reader, we provide the short proofs.
\begin{lemma}
  \label{lem:Gramian_inverse}
  Let $(f_n)_{n=1}^N$ be a finite sequence of linearly independent vectors in a Hilbert space
  and let $G = [ \langle f_i,f_j \rangle]_{i,j=1}^N$ be its Gramian.
  Let $(g_n)_{n=1}^N$ be the minimal dual system of $(f_n)_{n=1}^N$.
  Then $G^{-1} = [\langle g_i,g_j \rangle]_{i,j=1}^N$.
\end{lemma}

\begin{proof}
  Let $M = \bigvee_{n=1}^N f_n$ and let
  \begin{equation*}
    T: M \to \mathbb{C}^N, \quad x \mapsto ( \langle x,f_n \rangle)_{n=1}^N.
  \end{equation*}
  A simple computation shows that $T^* e_n = f_n$, where $(e_n)$ is the usual orthonormal basis of $\mathbb{C}^n$.
  By linear independence, $T$ is invertible, and $G = T T^*$.
  So $G^{-1} = (T^*)^{-1} T^{-1}$.
  By definition of minimal dual system, $T^{-1} e_n = g_n$.
  Another small computation shows that
  \begin{equation*}
    (T^{-1})^* x  = ( \langle x, g_n \rangle)_{n=1}^N
  \end{equation*}
  for all $x \in M$.
  Thus, $G^{-1} = (T^*)^{-1} T^{-1}$ is the Gramian of the $g_n$.
\end{proof}

If $A \in M_n(\mathbb{C})$, we denote the $(k,n)$ entry of $A$ by $A_{kn}$.
\begin{lemma}
  \label{lem:inverse_diagonal}
  Let $A \in M_n(\mathbb{C})$ be positive and invertible.
  Then for each $n=1,\ldots,N$, we have
  \begin{equation*}
    (A^{-1})_{nn} \ge (A_{nn})^{-1}.
  \end{equation*}
\end{lemma}

\begin{proof}
  Since $A$ is positive and invertible, it is the Gramian of a linearly independent sequence $(f_n)_{n=1}^N$ in $\mathbb{C}^N$.
  Let $(g_n)_{n=1}^N$ be the minimal dual system. Then Lemma \ref{lem:Gramian_inverse} and the Cauchy-Schwarz inequality show that
  \begin{equation*}
    (A^{-1})_{nn} = \|g_n\|^2 \ge \frac{1}{\|f_n\|^2} | \langle g_n, f_n \rangle|^2 = \frac{1}{\|f_n\|^2} =
    (A_{nn})^{-1}. \qedhere
  \end{equation*}
\end{proof}

We are now ready for the proof of Theorem \ref{thm:Feichtinger}.
The basic idea is the following.
Let $G$ be the Gramian of $(f_n)$. As explained above,
$G$ may be unbounded. Instead, we consider the operator $(I + G^{-1})^{-1}$, which, at least formally,
is a positive contraction. In particular, $(I + G^{-1})^{-1}$ is again a Gramian of a sequence $(h_n)$,
and we will apply the solution of the Feichtinger conjecture (Theorem \ref{thm:Feichtinger_Bessel}) to the
sequence $(h_n)$. To avoid technical difficulties in working with unbounded operators, we will first consider
finite sequences, and then pass to the limit.
  Of course, it is then crucial to obtain estimates that are independent of the length of the sequence.

\begin{proof}[Proof of Theorem \ref{thm:Feichtinger}]
  Let $(f_n)_{n=1}^N$ be an $\varepsilon$-uniformly minimal sequence and let $G = [ \langle f_i,f_j \rangle]_{i,j=1}^N $ be its Gramian. By uniform minimality, the $f_n$ are linearly independent, so $G$ is invertible,
  and $G^{-1}$ is positive. Thus, $H = (I + G^{-1})^{-1}$ is a positive contraction.
  In particular, $H$ is again a Gramian, so $H = [ \langle h_i,h_j \rangle]_{i,j=1}^N$ for some vectors $h_1,\ldots,h_n \in \mathbb{C}^N$.

  In order to apply Theorem \ref{thm:Feichtinger_Bessel}, we require a lower bound on the norm of the $h_n$. Since the $f_n$ are $\varepsilon$-uniformly minimal,
  Lemmas \ref{lem:dual_system} and \ref{lem:Gramian_inverse} imply that the diagonal entries of $G^{-1}$ are bounded above by $\frac{1}{\varepsilon^2}$, so the diagonal entries of $I + G^{-1}$ are bounded above by $1+\frac{1}{\varepsilon^2}$.
  Lemma \ref{lem:inverse_diagonal} then shows that the diagonal entries of $H$ are bounded below
  by $(1 + \varepsilon^{-2})^{-1}$. Therefore,
  \begin{equation*}
    \|h_n\|^2 \ge (1 + \varepsilon^{-2})^{-1} \ge \frac{1}{2} \varepsilon^2.
  \end{equation*}

  In this setting, Theorem \ref{thm:Feichtinger_Bessel} shows that there exist universal constants $M,\tau > 0$ so that $(h_n)_{n=1}^N$ can be partitioned into $r \le \frac{M}{\varepsilon^2}$ sequences, each of whose Gramian is bounded below by $\tau \varepsilon^2$.
  Now, since $H = (I + G^{-1})^{-1}$, we have $H \le G$.
  Thus, the Gramian of any subsequence of $(f_n)$ is bounded below by the Gramian of the corresponding
  subsequence of $(h_n)$. In particular, if we partition $(f_n)$ in the same
  way as $(h_n)$, then we obtain a partition of $(f_n)$ into $r$ sequences, each of whose Gramian
  is bounded below by $\tau \varepsilon^2$.

  Since the number $r$ is independent of $N$, the case of an infinite sequence follows from the case of finite
  sequences by a combinatorial compactness result known as the pinball principle; see \cite[Proposition 2.1]{CCL+05} or \cite[Theorem 6.9]{BCM+15}
  and also the proof of \cite[Theorem 6.11]{BCM+15}.
\end{proof}

\section{Equivalence of simple interpolation and strong separation}
\label{sec:si_ss}

We are now ready for the proof of Theorem \ref{thm:simply_interpolating_intro}.
As mentioned before, it is known that every simply interpolating sequence is strongly separated;
see the discussion following Proposition \ref{prop:SS_Pick}.
We will establish the following quantitative converse.

\begin{theorem}
  \label{thm:simply_interpolating}
  Let $\mathcal{H}$ be a normalized complete Pick space.
  Then every strongly separated sequence is simply interpolating.
  More precisely, if $(z_n)$ is $\varepsilon$-strongly separated, then the Gramian
  of the normalized reproducing kernels is bounded below by $\frac{\tau}{M} \varepsilon^6$,
  where $\tau$ and $M$ are the constants of Theorem \ref{thm:Feichtinger}.
\end{theorem}

\begin{proof}
  Let $(z_n)$ be sequence in $X$ that is $\varepsilon$-strongly separated.
  Clearly, $\varepsilon \le 1$.
  Let $(f_n)$ be the normalized reproducing kernels, which are $\varepsilon$-uniformly minimal
  by Proposition \ref{prop:SS_Pick}.
  By Theorem \ref{thm:Feichtinger}, we may partition $(f_n)$ into $r = \lfloor \frac{M}{\varepsilon^2} \rfloor$
  sequences, each of whose Gramian is bounded below by $\tau \varepsilon^2$.
  Let $Z_1,\ldots,Z_r$
  be the corresponding partition of the sequence $(z_n)$. We also write $Z = \{z_n: n \in \mathbb{N}\}$.
  
  We can now argue similarly to the proof of the main result of \cite{AHM+17}.
  Here are the details.
  For each $j=1,\ldots,r$, Proposition $\ref{prop:row}$ yields a family of multipliers $(\varphi_z)_{z \in Z_j}$
  that form a row of multiplier norm at most one with $\varphi_z(w) = \tau^{1/2} \varepsilon \delta_{zw}$
  for all $z,w \in Z_j$.
  Let $\Phi$ be the row of all the $\varphi_z$ for $z \in Z$, in the order in which the points
  $z \in Z$ appear in the sequence $(z_n)$.
  Then $\Phi$ has multiplier norm at most $\sqrt{r}$ and the
  $n$-th entry of $\Phi(z_n)$ is $\tau^{1/2} \varepsilon$.
  By strong separation, there exists a diagonal multiplier $D$ of norm at most $1$ such that $D(z_n)$
  is the matrix whose $(n,n)$ entry is $\varepsilon$ and whose other entries are all zero.
  Let $\Psi = r^{-1/2} \Phi D$. Then $\Psi$ is a row multiplier of norm at most $1$
  such that $\Psi(z_n) = r^{-1/2} \tau^{1/2} \varepsilon^2 e_n^T$ for all $n$.
  Again by Proposition \ref{prop:row}, we find that the Gramian is bounded below by
  $r^{-1} \tau \varepsilon^4 \ge \frac{\tau}{M} \varepsilon^6$.
\end{proof}

\begin{remark}
  The analogue of Theorem \ref{thm:simply_interpolating} for general uniformly minimal sequences
  of unit vectors in a Hilbert space fails.
  That is, Gramians of uniformly minimal sequences need to be bounded below.

  For example, consider $G = (1 + \frac{1}{n}) I - P$,
  where $P$ is the matrix all of whose entries are $\frac{1}{n}$. Since $P$ is an orthogonal projection, $G$ is positive and invertible, and so $G$ is the Gramian of a sequence of unit vectors.
  A little computation shows that
  \begin{equation*}
    G^{-1} = \frac{n}{n+1} (I + n P),
  \end{equation*}
  so that the diagonal entries of $G^{-1}$ are at most $2$, but $\|G^{-1}\| \ge n$.
  Taking a direct sum of such matrices, we obtain a uniformly minimal sequence whose Gramian is not bounded below. Note that the Gramian is bounded in this example.
\end{remark}

  More relevantly for the subject of this article, one can ask which reproducing kernel Hilbert spaces
  have the property that any uniformly minimal sequence of normalized reproducing kernels has a Gramian that is bounded below.
  Theorem \ref{thm:simply_interpolating} shows that this is the case for normalized complete Pick spaces.
  Moreover, a theorem of Schuster and Seip \cite{SS98} shows that this is also true in the Bergman space on the disc; see also \cite[Section 5]{KS}.
  On the other hand, Schuster and Seip \cite[Section 4]{SS00} show that the Paley--Wiener space does not have this property;
  see also \cite[Theorem 3.1]{AH10}.

Theorem \ref{thm:simply_interpolating} yields a different proof of a theorem of Berndtsson \cite{Berndtsson85},
according to which every sequence in $\mathbb{B}_d$ that is uniformly separated for $H^2_d$ is interpolating for $H^\infty(\mathbb{B}_d)$.
Indeed, this now follows from the fact that uniform separation for $H^2_d$ implies strong separation for $H^2_d$
(see the discussion following Proposition \ref{prop:uniformly_separated})
and the next corollary.

\begin{corollary}
  \label{cor:Hinf}
  Every sequence in $\mathbb{B}_d$ that is strongly separated for $H^2_d$ is interpolating for $H^\infty(\mathbb{B}_d)$.
\end{corollary}

\begin{proof}
  Let $(z_n)$ be a strongly separated sequence for $H^2_d$.
  Then $(z_n)$ is simply interpolating for $H^2_d$ by Theorem \ref{thm:simply_interpolating}.
  By Proposition \ref{prop:row}, there exists a bounded row multiplier $\Psi = (\psi_1,\psi_2,\psi_3,\ldots) \in \Mult(H^2_d \otimes \ell^2, H^2_d)$
  with $\psi(z_n) = e_n^T$. In particular, $ \sup_{z \in \mathbb{B}_d} \sum_{n} |\psi_n(z)|^2 < \infty$.
  Thus,
  \begin{equation*}
    \ell^\infty \to H^\infty(\mathbb{B}_d), \quad (\alpha_n) \mapsto \sum_{n} \alpha_n \psi_n^2,
  \end{equation*}
  is a linear operator of interpolation.
\end{proof}

Conversely, Berndtsson showed in \cite{Berndtsson85} that uniform separation with respect to $H^2_d$ is not necessary for interpolation
with respect to $H^\infty(\mathbb{B}_d)$.
In Section \ref{sec:ball}, we will construct a  sequence that is strongly separated for $H^2_d$, but not uniformly separated for $H^2_d$.
By Corollary \ref{cor:Hinf}, this sequence will then also by an interpolating sequence for $H^\infty(\mathbb{B}_d)$ that is not uniformly separated
for $H^2_d$.

  As mentioned in the introduction, the implication (WS) + (C) $\Rightarrow$ (IS) can also be proved
  without using the positive resolution of the Feichtinger conjecture, by making use of the column-row property \cite{Hartz20}.
  Thus, we ask:

\begin{question}
  Can Theorem \ref{thm:simply_interpolating} be proved without the use of the positive resolution of the Feichtinger conjecture?
\end{question}

\section{Uniformly separated sequences that are not interpolating}
\label{sec:us_not_is}

The goal of this section is to construct sequences that are uniformly separated, but not interpolating in $\mathcal{D}_a$ for $0 < a < 1$.
Throughout, we let
\begin{equation*}
  k_a(z,w) = \frac{1}{(1 - z \overline{w})^a} \quad (z,w \in \mathbb{D})
\end{equation*}
denote the reproducing kernel of $\mathcal{D}_a$.
Given a subset $Z = \{z_1,z_2,\ldots\} \subset \mathbb{D}$, we consider
the Gramian $G = (G_{kj})$ of the normalized reproducing kernels at the points in $Z$,
given by
\begin{equation*}
  G_{kj} = \frac{k_a(z_k,z_j)}{k_a(z_k,z_k)^{1/2} k_a(z_j,z_j)^{1/2}}.
\end{equation*}
We also call $G$ the $\mathcal{D}_a$-Gramian of the points in $Z$.

In a first step, we construct finite sequences for which the operator norm of the Gramian is much larger than the
$\ell^2$-norms of the columns of the Gramian.
The construction uses equidistributed points on circles centered at the origin.
As usual, if $A$ and $B$ are non-negative real numbers depending on some parameters, we write $A \lesssim B$ to mean that $A \le C B$ for some constant $C \ge 0$ that does not depend on the parameters, and $A \lesssim_a B$ to indicate that the constant $C$ may depend on the parameter $a$. We also write $A \simeq B$ to mean that $A \lesssim B$ and $B \lesssim A$,
and similarly $A \simeq_a B$ to indicate that the implied constants may depend on $a$.

\begin{lemma}
  \label{lem:equidistributed}
  Let $0 <a \le 1$, let $N \in \mathbb{N}$, $N \ge 2$ and let $1 - \frac{1}{N} \le r < 1$. Define
  \begin{equation*}
    Z_N = \{ r e^{\frac{2 \pi i n}{N}}: 0 \le n < N \}.
  \end{equation*}
  Let $G_N = (G_{kj})$ be the $\mathcal{D}_a$-Gramian of the points in $Z_N$.
  Then:
  \begin{enumerate}[label=\normalfont{(\alph*)}]
    \item 
   The points in $Z_N$ are weakly separated, with weak separation constant independent of $r$ and $N$.
 \item
   The following asymptotic relation holds:
      \[
        \sup_{j}
        \sum_{n\neq j}|G_{nj}|^2\simeq_{a}\begin{cases}
          (1-r)^{2a}N,\qquad&\text{if } 0 < a < \frac{1}{2}, \\
          (1-r)N\log N, \qquad &\text{if }a= \frac{1}{2}, \\
          (1-r)^{2 a} N^{2a}, \qquad& \text{if } \frac{1}{2}< a \le 1.\\
\end{cases}
\]
\item We have the lower bound
  $\|G_N\| \gtrsim_a N (1 - r)^a$.
  \end{enumerate}
\end{lemma}

\begin{proof}
  We will use the estimate 
  \begin{equation*}
    |1 - z \overline{w}| \simeq \max \Big\{ 1 - |z|, 1 - |w|, \Big| \frac{z}{|z|} - \frac{w}{|w|} \Big| \Big\} \quad (z,w \in \overline{\mathbb{D}}\setminus\{ 0 \}).
  \end{equation*}
When $w=1$ the above estimate is an exercise in plane geometry. For general points $z,w \in \overline{\mathbb{D}}\setminus \{ 0 \}$ it follows from the previous case and the estimate
\[ \max\{1-|z|, 1-|w|\}\leq 1-|z w| \leq 2 \max\{ 1-|z|, 1-|w| \}.  \]
  If $1 \le j \le \lfloor \frac{N}{2} \rfloor$, then
  $|1 - e^{\frac{2 \pi i j}{N}}| \simeq \frac{j}{N}$, so since $1 -r \le \frac{1}{N}$,
  it follows that
  \begin{equation}
    \label{eqn:estimate_equidist}
    |1 - r^2 e^{\frac{2 \pi i j}{N}}| \simeq \frac{j}{N} \quad (1 \le j \le \lfloor \frac{N}{2} \rfloor).
  \end{equation}

  (a) Let $d_a$ be the metric induced by the kernel $k_a$.
  Then
  \[
    d_a^2(r e^{\frac{2 \pi i n}{N}}, r e^{\frac{2 \pi i j}{N}}) = 1 - |G_{n j}|^2,
  \]
  and
  \begin{equation}
  \label{eqn:circular_gramian}
    |G_{nj}|^2 = \frac{|1 - r^2|^{2 a}}{|1 - r^2 e^{\frac{2 \pi i (n-j)}{N}}|^{2 a}}.
  \end{equation}
  Thus, the minimal $d_a$-distance between two points in $Z$ is $d_a(r, r e^{\frac{2 \pi i}{N}})$.
  The metric $d_a$ is equivalent to the metric $d_1$, which in turn is the pseudohyperbolic metric $\rho$,
  given by
  \begin{equation*}
    \rho(z,w) = \Big| \frac{z - w}{1 - z \overline{w}} \Big|.
  \end{equation*}
  Since
  \begin{equation*}
    \rho(r, r e^{\frac{2 \pi i}{N}}) = \frac{r|1 - e^{\frac{2 \pi i}{N}}|}{|1 - r^2 e^{\frac{2 \pi i}{N}}|}
  \end{equation*}
  is bounded below by \eqref{eqn:estimate_equidist}, part (a) follows.
  
  (b)
  Using Equation \eqref{eqn:circular_gramian}, we find that
  \begin{equation*}
    \sup_{j} \sum_{n \neq j} |G_{nj}|^2 = \sum_{n=1}^{N-1} |G_{n0}|^2 \simeq \sum_{n=1}^{\lfloor \frac{N}{2} \rfloor} |G_{n 0}|^2.
  \end{equation*}
  Relation \eqref{eqn:estimate_equidist} implies that
  \begin{equation*}
    |G_{n0}|^2 \simeq_a (1 - r^2)^{2 a} N^{2 a} \frac{1}{n^{2 a}},
  \end{equation*}
  so
  \begin{equation*}
    \sup_{j} \sum_{n \neq j} |G_{nj}|^2 \simeq_a (1 - r^2)^{2 a} N^{2 a}
    \sum_{n=1}^{\lfloor \frac{N}{2} \rfloor}  \frac{1}{n^{2 a}},
  \end{equation*}
  from which the first claim follows.

  (c) Let $z_n = r e^{\frac{2 \pi i n}{N}}$ and recall from Proposition \ref{prop:BG} that
  \begin{equation*}
    \|G_N\| = \sup \Big\{ \sum_{n=0}^{N-1} \frac{|f(z_n)|^2}{k_a(z_n,z_n)} : \|f\|_{\mathcal{D}_a} \le 1 \Big\}
  \end{equation*}
  The lower bound for $\|G_N\|$ follows by choosing $f=1$.
\end{proof}

By choosing appropriate values of $r$ and $N$ depending on $a$ in Lemma \ref{lem:equidistributed},
we can find finite sequences for which the $\ell^2$ norms of the columns of the Gramian are uniformly bounded, but
the operator norm of the Gramian tends to infinity.
We will use biholomorphic automorphisms of the disc to assemble the finite sequences
into one infinite sequence.
Since we will require it later, we directly work on the unit ball.
Let $\Aut(\mathbb{B}_d)$ denote the group of biholomorphic automorphisms of $\mathbb{B}_d$.
For background material on automorphisms, see \cite[Section 2.2]{Rudin1980}.
Let
\begin{equation*}
  K_a(z,w) = \frac{1}{(1 - \langle z,w \rangle)^a } \quad (z,w \in \mathbb{B}_d).
\end{equation*}

\begin{lemma}
  \label{lem:automorphism}
  Let $0 < a \le 1$, let $z_1, \ldots,z_N,w_1,\ldots,w_M \in \mathbb{B}_d$ and let $\varepsilon > 0$.
  Then there exists $\varphi \in \Aut(\mathbb{B}_d)$ such that
\[
\sum_{n=1}^N\sum_{m=1}^M\frac{|K_a(\varphi(z_n), w_m)|^2}{K_a(\varphi(z_n), \varphi(z_n))K_a(w_m, w_m)}<\varepsilon.
\]
\end{lemma}

\begin{proof}
  Since $M,N$ are fixed and since the function $t \mapsto t^a$ is continuous at $0$,
  it suffices to show that for each $\delta > 0$, there exists $\varphi \in \Aut(\mathbb{B}_d)$
  with
  \begin{equation*}
    \frac{|K_1(\varphi(z_n), w_m)|^2}{K_1(\varphi(z_n), \varphi(z_n))K_1(w_m, w_m)} \le \delta \quad \text{ for all } 1 \le n \le N, 1 \le m \le M.
  \end{equation*}
  To this end, let
  \[
    \rho(z, w)=\sqrt{1-\frac{|K_1(z, w)|^2}{K_1(z, z)K_1(w, w)}} \quad (z, w\in\B_d)
\]
be the pseudo-hyperbolic distance on $\B_d$, and let
\begin{equation*}
\beta(z, w)=\frac{1}{2}\log\frac{1+\rho(z, w)}{1-\rho(z, w)} \quad (z, w\in\B_d)
\end{equation*}
be the Bergman metric. Then $\beta$ is a metric, which is automorphism
invariant in the sense that
$\beta(\varphi(z),\varphi(w)) = \beta(z,w)$ for all $z,w \in \mathbb{B}_d$
and all $\varphi \in \Aut(\mathbb{B}_d)$. Moreover,
for each fixed $z \in \mathbb{B}_d$, we have $\beta(z,w) \to \infty$ as $\|w\| \to 1$.
These facts can be found for instance in \cite{Zhu05}, see Lemma 1.2 and Section 1.5 there.

Now, let $R = \max \{ \beta(0,w_m): 1 \le m \le M \}$. We may find $x \in \mathbb{B}_d$ with
\begin{equation*}
  \beta(z_n,x) \ge R + \frac{1}{2} \log \frac{4}{\delta} \quad \text{ for all } 1 \le n \le N.
\end{equation*}
Let $\varphi \in \Aut(\mathbb{B}_d)$ be an automorphism that maps $x$ to $0$. Then by the reverse triangle inequality for $\beta$,
we find that
\begin{equation*}
  \beta(\varphi(z_n),w_m) \ge \beta(\varphi(z_n),0) - \beta(0, w_m)
  = \beta(z_n,x) - \beta(0,w_n) \ge \frac{1}{2} \log \frac{4}{\delta}.
\end{equation*}
Since $\beta \le \frac{1}{2} \log \frac{4}{1 - \rho^2}$, this estimate shows that
\begin{equation*}
  \frac{1}{1- \rho^2(\varphi(z_n),w_m)} \ge \frac{1}{4} e^{2 \beta(\varphi(z_n),w_m)} \ge \frac{1}{\delta},
\end{equation*}
from which the claim follows.
\end{proof}

The next lemma now allows us to assemble finite sequences into one infinite sequence.

\begin{lemma}
  \label{lem:assembly}
  Let $0 < a \le 1$. For $N \in \mathbb{N}$, let $Z_N$ be a finite sequence in $\mathbb{B}_d$
  and let $G_N$ be the Gramian of the points in $Z_N$ with respect to the kernel $k_a$.
  Let $\varepsilon > 0$.
  
  Then, for all $N \in \mathbb{N}$, there exists $\varphi_N \in \Aut(\mathbb{B}_d)$ such that the Gramian of the points in $\bigcup_{N \in \mathbb{N}} \varphi_N(Z_N)$ has the form
  \begin{equation*}
    G =  D \Big( \bigoplus_{N \in \mathbb{N}} G_N \Big) D^* + R,
  \end{equation*}
  where $D$ is a diagonal matrix with unimodular entries, and $R$ is a Hilbert--Schmidt operator
  with $\|R\|_{\operatorname{HS}} < \varepsilon$. Moreover, the only non-zero entries of $R$ occur
  outside of the block diagonal corresponding to the block diagonal operator $\bigoplus_{N \in \mathbb{N}} G_N$.
\end{lemma}

\begin{proof}
  We choose the automorphisms $\varphi_N$ inductively using Lemma \ref{lem:automorphism}.
  Let $\varphi_1$ be the identity.
 If $\varphi_{1},\ldots,\varphi_{N-1}$ have been chosen,
let $W_N = \bigcup_{n=1}^{N-1} \varphi_n(Z_n)$, and
apply Lemma \ref{lem:automorphism} to find $\varphi = \varphi_N \in \Aut(\mathbb{B}_d)$ such that
\begin{equation}
  \label{eqn:off_diagonal}
  \sum_{z \in Z_N} \sum_{w \in W_N} \frac{|K_a(\varphi(z), w)|^2}{K_a(\varphi(z),\varphi(z)) K_a(w,w)} < \varepsilon^2 2^{-N-1}.
\end{equation}
Let $\widetilde{G}_N$ be the Gramian of the points in $\Phi_N(G_N)$.
Then the Gramian $G$ of the points in $\bigcup_{N \in \mathbb{N}} \varphi_N(Z_N)$ is of the form
\begin{equation*}
  G = \bigoplus_{N \in \mathbb{N}} \widetilde{G}_N + R,
\end{equation*}
where, thanks to \eqref{eqn:off_diagonal},
\begin{equation*}
  \|R\|_{\operatorname{HS}}^2 = 2 \sum_{N=1}^\infty \sum_{z \in \varphi_N(Z_N)} \sum_{w \in W_N} \frac{|K_a(z,w)|^2}{K_a(z,z) K_a(w,w)} < \varepsilon^2.
\end{equation*}

Finally, a well-known identity for ball automorphisms (see \cite[Theorem 2.2.5]{Rudin1980})
shows that
\begin{equation*}
  \frac{(1 - \|\varphi(z)\|^2)^{1/2} (1 - \|\varphi(w)\|^2)^{1/2}}{1 - \langle \varphi(z), \varphi(w) \rangle }
  = \omega(z) \overline{\omega(w)}   \frac{(1 - \|z\|^2)^{1/2} (1 - \|w\|^2)^{1/2}}{1 - \langle z, w \rangle },
\end{equation*}
where
\begin{equation*}
  \omega(z) = \frac{1 - \langle z,\varphi^{-1}(0) \rangle }{|1 - \langle z,\varphi^{-1}(0) \rangle| }.
\end{equation*}
From this and from the fact that $(s t)^a = s^a t^a$ if $s$ and $t$ both have positive real part, it follows that
\begin{equation*}
  \frac{K_a(\varphi(z),\varphi(w))}{K_a(\varphi(z),\varphi(z))^{1/2} K_a(\varphi(w),\varphi(w))^{1/2}}
  = \omega(z)^a \overline{\omega(w)^a}
  \frac{K_a(z,w)}{K_a(z,z)^{1/2} K_a(w,w)^{1/2}}.
\end{equation*}
Since $\omega^a$ takes unimodular values, we see that
$G_N$ and $\widetilde{G}_N$ are unitarily equivalent via a diagonal matrix with unimodular entries.
\end{proof}

We are now ready to construct a uniformly separated sequence that is not interpolating in $\mathcal{D}_a$.
This establishes Theorem \ref{thm:counterexamples_intro} (a) for $\mathcal{D}_a$.

\begin{proposition}
  \label{prop:D_a_unif_sep_not_int}
  Let $0 < a < 1$. There exists a sequence $(z_n)$ in $\mathbb{D}$ that is uniformly separated, but not interpolating for $\mathcal{D}_a$.
\end{proposition}

\begin{proof}
  Let $N \in \mathbb{N}$ with $N \ge 2$. Set
\begin{equation}
\label{eqn:rnDa}
r_N = \begin{cases}
  1-\frac{1}{N^\frac{1}{2a}},\qquad& \text{if } 0 < a< \frac{1}{2}, \\
  1-\frac{1}{N\log N},\qquad & \text{if } a=1/2,\\
  1-\frac{1}{N},\qquad& \text{if } \frac{1}{2} < a < 1,
\end{cases}
\end{equation}
and define
\begin{equation*}
  Z_N = \{ r_N e^{\frac{2 \pi i n}{N}} : 0 \le n < N-1 \}.
\end{equation*}
Let $G_N$ denote the $\mathcal{D}_a$-Gramian of the points in $Z_N$.
The Lemma \ref{lem:equidistributed} shows that the $\ell^2$ norms of the columns of $G_N$ are uniformly bounded in $N$,
but $\|G_N\| \to \infty$ as $N \to \infty$.

To assemble the finite sequences $(Z_N)$ into one sequence, we
apply Lemma \ref{lem:assembly} with $\varepsilon = \frac{1}{2}$ and consider the points in $Z = \bigcup_{N \in \mathbb{N}} \varphi(Z_N)$. Then then Gramian $G$ of the points in $Z$ is unbounded, but the $\ell^2$-norms of the columns
of $G$ are uniformly bounded.

Finally, it remains to check weak separation, i.e.\ that there exists $\delta > 0$ with $|G_{kj}| \le 1 - \delta$ for all $k \neq j$, where $G = (G_{kj})$. If $w \in \varphi_N(Z_N)$ and $z \in \varphi_N(Z_M)$ for $N \neq M$, then the modulus of the corresponding entry of $G$ is uniformly bounded away from $1$ by our
choice $\varepsilon = \frac{1}{2}$.
If $w,z \in \varphi_N(Z_N)$ with $w \neq z$, then the modulus of the corresponding entry of $G$ is uniformly bounded away from $1$ by Lemma \ref{lem:equidistributed} (a).
Uniform separation then follows from Proposition \ref{prop:uniformly_separated}.
\end{proof}

\section{Simply interpolating sequences that generate an infinite measure}
\label{sec:si_im}

The goal of this section is to construct examples of strongly separated sequences in $\mathcal{D}_a$ that generate an infinite
measure. In particular, these sequences are not uniformly separated by Propositions \ref{prop:uniformly_separated} and \ref{prop:US_FM}.
For $0 < a \le \frac{1}{2}$, strongly separated sequences that are not uniformly separated can also be constructed
by a procedure similar to the one in the preceding section, as the next example shows.
In this example, simple interpolation can be checked directly without appealing to Theorem \ref{thm:simply_interpolating}.
\begin{example}
\label{example:circulant}
 Let $0<a\leq1/2$, and choose, for all $N$ in $\N$, $N$ equi-distributed points on the circle of radius $r_N=1-1/N$:
 \[
  Z_N = \left\{ \left(1-\frac{1}{N}\right) e^{\frac{2 \pi i n}{N}} : 0 \le n < N \right\}.
 \]
 Let $G_N$ be their Gramian with respect to the kernel $k_a$. By Lemma \ref{lem:equidistributed} (b), the supremum of the $\ell^2$-norms of the columns of $G_N$ tend to infinity as $N$ tends to infinity.
 
 On the other hand, one can estimate directly that the $G_N$ are uniformly bounded below thanks to equi-distribution and the fact that $k_a$ is invariant under rotations. Indeed, such symmetry implies that $G$ is a \emph{circulant} matrix, that is, the $n$-th column of $G$  is $\sigma^n(v)$, where $v$ is the first column of $G_N$ and $\sigma$ is the (cyclic) shift of the coordinates in $\mathbb{R}^N$.
 The eigenvalues of the self-adjoint circulant matrix $G_N$ can be computed as
 \[
 \lambda_j=P(w^j), \qquad j=0\dots, N-1
 \]
 where $w=e^\frac{2\pi i}{N}$ and $P(z)=\sum_{n=0}^{N-1}G_{n0}z^n$ is the polynomial whose ordered coefficients are the coordinates of the first column of $G$; see for instance \cite{circulant}.
 Thus,
 \[
  \lambda_j = k_a(r_N,r_N)^{-1} \sum_{n=0}^{N-1} k_a(r_N,r_N w^n) w^{nj}.
 \]
Writing $k_a(z, w)=\sum_{m=0}^\infty c_m (z\overline{w})^m$,
 where $c_m = \binom{-a}{m}$ and using the standard orthogonality relations
 for $N$-th roots of unity, one has
 \[
  \sum_{n=0}^{N-1}k_a(r_N, r_Nw^n)w^{nj}
  = \sum_{m=0}^\infty c_m r_N^{2m} \sum_{n=0}^{N-1} w^{n (j-m)}
  = N \sum_{m=0}^\infty c_{Nm +j} r_N^{2(Nm + j)}.
 \]
Since $a<1$, the coefficients $(c_m)_{m}$ are decreasing, from which
we conclude that
\[
\lambda_{N-1} < \lambda_{N-2} < \cdots < \lambda_0.
\]
(We mention in passing that $\|G_N\| = \lambda_0 \ge N k_a(r_N,r_N)^{-1} \simeq_a N^{1-a}$, so that we recover Lemma \ref{lem:equidistributed} (c) in this case.)
Since $c_m \simeq_a (m+1)^{a-1}$, we may estimate the least eigenvalue $\lambda_{N-1}$ of $G_N$ as
\begin{align*}
    \lambda_{N-1} &= N k_a(r_N,r_N)^{-1} \sum_{m=0}^\infty c_{N m + N -1} r_N^{2( N m + N-1)} \\
&\simeq_a~{N(1-r_N)^{a}}\sum_{m=0}^\infty(N(m+1))^{a-1}r_N^{2N(m+1)}\\
&=\sum_{m=0}^\infty (m+1)^{a-1} r_N^{2N(m+1)} \\
&\simeq_a \frac{r_N^{2N}}{(1 - r_N^{2N})^a}.
\end{align*}
Since $r_N^N \to e^{-1}$ as $N \to \infty$, it follows that there exists $\varepsilon > 0$ such that $G_N$ is bounded below by $2 \varepsilon$ for all $N$.

  Given such $\varepsilon$, apply Lemma \ref{lem:assembly} and assemble the finite point configurations into one sequence $Z:=(\varphi_N(Z_N))_{N\in\N}$ having Gramian $G$. Since the off-diagonal blocks of $G$ have Hilbert-Schmidt norm less than $\varepsilon$, $G$ is bounded below by $\varepsilon$, and we obtain a sequence that is simply interpolating (and hence strongly separated) but that is not uniformly separated, since thanks to Lemmas \ref{lem:equidistributed} (b) and \ref{lem:assembly} the $l^2$ norms of the columns of the $N\times N$ diagonal blocks of $G$ are unbounded in $N$. Observe that each column of $G$ is bounded, thus $Z$ generates a finite measure. Therefore, this proves Theorem \ref{thm:counterexamples_intro} (b) for the spaces $\mathcal{D}_a$, $0<a\leq1/2$.
\end{example}
In order to cover the whole range $0 < a < 1$ and in order to obtain infinite measure sequences, we will pursue a different approach.

A Fuchsian group is a discrete subgroup of $\operatorname{Aut}(\mathbb D)$, the group of biholomorphic automorphisms of the unit disc.
We will consider sequences that are orbits of Fuchsian groups.
Background material on Fuchsian groups can be found in \cite{DalBo11,Katok92}.

A proper subset $Z \subsetneq \mathbb{D}$ is said to be a zero set for $\mathcal{D}_a$ if there exists a function $f \in \mathcal{D}_a \setminus \{0\}$
such that $Z = \{z \in \mathbb{D}: f(z) = 0 \}$.
It is a property of $\mathcal{D}_a$ that subsets of zero sets are zero sets,
see e.g.\ \cite[Proposition 9.37]{AM02}.

\begin{proposition}
    \label{prop:Fuchsian_zero}
    Let $0 < a \le 1$. Let $G$ be a Fuchsian group, let  $w \in \mathbb D$ and let $Z$ be the orbit of $w$ under $G$.
    The following are equivalent:
    \begin{enumerate}
    \item[(i)] $Z$ is a zero set for $\mathcal D_{a}$;
    \item[(ii)] $Z$ is strongly separated for $\mathcal D_a$.
    \end{enumerate}
\end{proposition}

\begin{proof}
  (ii) $\Rightarrow$ (i) By assumption, there exists a non-zero $f \in \mathcal{D}_a$ that vanishes on $Z \setminus \{w\}$.
  Then $(z-w) f(z) \in \mathcal{D}_a$ vanishes on $Z$. So $Z$ is a subset of a zero set, and hence a zero set.

    (i) $\Rightarrow$ (ii) Let $Z$ be a zero set. Since subsets of zero sets are zero sets, there exists a multiplier $\varphi_w \in \Mult(\mathcal D_a)$ of norm at most one vanishing on $Z \setminus \{w\}$
    such that $\varphi_w(w) = \varepsilon > 0$.
    It is well known that the multiplier algebra of $\mathcal{D}_a$ is isometrically invariant under biholomorphic automorphisms; this follows from a computation with the reproducing kernels, see e.g.\ \cite[Corollary 4.4]{Hartz17a}.
Thus, the family $\{\varphi_w \circ \theta: \theta \in G \}$
    is a family of multipliers of norm at most one, which shows that $Z$ is strongly separated.
\end{proof}

\begin{example}
  \label{exa:infinite_measure}
    Let $Z \subset \mathbb D$ be the Cayley transform of the sequence $(n+i)_{n \in \mathbb Z}$ in the upper half plane; explicitly,
    \begin{equation*}
      Z = \Big\{ \frac{n}{n+ 2 i}: n \in \mathbb{Z} \Big\}.
    \end{equation*}
    Then $Z$ is the orbit of a Fuchsian group (the group generated by the disc automorphism corresponding to the upper half plane automorphism $z \mapsto z +1$).
    Since
    \begin{equation*}
      1 - \Big| \frac{n}{n+2 i} \Big| \simeq 1 - \Big| \frac{n}{n+2 i}\Big|^2
     =  \frac{4}{n^2+4},
    \end{equation*}
    we find that
    \[
      \sum_{z \in Z} (1 - |z|)^a < \infty \quad \Longleftrightarrow \quad a > \frac{1}{2}.
    \]
    In particular, $Z$ is a Blaschke sequence that generates an infinite measure in $\mathcal D_a$ for $a \le \frac{1}{2}$. Moreover, $Z$ is contained in a horocycle touching $\partial \mathbb D$ at $1$.
    This horocycle has in particular finite order of contact with the boundary.
    In this setting, a result of Bogdan \cite{Bogdan96}, based on a theorem of Taylor and Williams \cite{TW71}, implies that $Z$ is a zero sequence
    for the classical Dirichlet space, and in particular for $\mathcal D_a$. Hence for $a \le \frac{1}{2}$,
    the sequence $Z$ is a strongly separated infinite measure sequence for $\mathcal D_a$. 
\end{example}

Using deeper results from the theory of Fuchsian groups, the construction in the previous example
can be extended to the whole range $0 < a< 1$.
This establishes Theorem \ref{thm:counterexamples_intro} (c) for $\mathcal{D}_a$.

\begin{proposition}
  \label{prop:DA_infinite_measure}
    Let $0< a < 1$. There exists a strongly separated sequence that generates an infinite measure in $\mathcal D_a$.
\end{proposition}

\begin{proof}
    A theorem of Beardon \cite[Theorem 1]{Beardon71} shows that there exists $\varepsilon > 0$ such that
    if $G$ denotes the Fuchsian group generated by the upper half plane automorphisms
    \[
        z \mapsto z + 2 + \varepsilon \quad \text{ and } z \mapsto - \frac{1}{z},
    \]
    and $Z$ is the orbit of any point under $G$,
    then
    \[
    \sum_{z \in Z} (1 - |z|)^a = \infty.
    \]
    (See e.g.\ the introduction of \cite{K.V.RAJESWARA69} for an explanation for why the convergence
    of the sum considered in \cite{Beardon71} is equivalent to the convergence of the sum above.)

    By Proposition \ref{prop:Fuchsian_zero}, we may finish the proof by showing that $Z$ is a zero set for the classical Dirichlet space. To this end, we use the fact, established by Beardon \cite{Beardon71}, that $G$ is a Fuchsian group of the second kind. Pommerenke \cite{Pommerenke76} showed
    that the limit set of any finitely generated Fuchsian group of the second kind is a Carleson set;
    see also \cite{Shirokov90,Vasin03} for an even stronger statement.
    A result of Daase \cite[Lemma 3]{Daase83}, combined with a theorem of Taylor on Williams \cite{TW71},
    then shows that $Z$ is the zero set of a function $f$ with the property that all derivatives of $f$ belong to the disc algebra. In particular, $Z$ is a zero set for the classical Dirichlet space.
\end{proof}

\section{Interpolating sequences in the ball}
\label{sec:ball}

The goal of this section is to construct examples of sequences on the ball
that distinguish the various interpolation and separation conditions for the Drury--Arveson space.
To this end, we will leverage our examples on the disc.

We use a device that is frequently useful for constructing examples related to the Drury--Arveson space;
it already appeared in Arveson's work, see \cite[Theorem 3.3]{Arveson98}.
Let
\[
r: \mathbb B_2 \to \mathbb D, \quad (z_1,z_2) \mapsto 2 z_1 z_2.
\]
A proof of the following result can be found in
\cite[Section 7]{AHM+20a}.

\begin{lemma}
  \label{lem:embedding}
  The map
\[
  \mathcal{D}_{1/2} \to H^2_2, \quad f \mapsto f \circ r,
\]
is an isometry, and the map
\[
 \Mult(\mathcal H) \to \Mult(H^2_2), \quad \varphi \mapsto \varphi \circ r,
\]
is a (complete) isometry.
\end{lemma}

The idea is now is to construct examples for $H^2_d$ from our examples from $\mathcal{D}_{1/2}$.
The following lemma is a first step in this direction.

\begin{lemma}
  \label{lem:US_ball}
        Let $z \in \mathbb D$, let $z^{1/2}$ be any square root of $z$, let $N \in \mathbb{N}, N \ge 2$, and define
    \[
      W = \{ 2^{-1/2} z^{1/2} ( e^{\frac{2 \pi i j}{N}}, e^{\frac{-2 \pi i j }{N}}) : 0 \le j \le N-1 \}.
    \]
    If $N \le \frac{1}{(1 - |z|)^{1/2}}$, then $W$ is uniformly separated for $H^2_2$ with uniform
    separation constant independent of $z$ and $N$.
\end{lemma}

\begin{proof}
  Let $ w_j = 2^{-1/2} z^{1/2} ( e^{\frac{2 \pi i j}{N}}, e^{\frac{-2 \pi i j }{N}})$.
  Let $G = (G_{nj})$ be the $H^2_2$-Gramian of the points $w_0,\ldots,w_{N-1}$. Thus,
  \begin{equation*}
    G_{n j} = \frac{(1 - \|w_n\|^2)^{1/2} (1 - \|w_j\|^2)^{1/2}}{1 - \langle w_n,w_j \rangle}
    = \frac{1 - |z|}{1 - |z| \cos( 2 \pi (n-j)/N)}.
  \end{equation*}

  First, we will show that $W$ is weakly separated with weak separation constant independent of $z$ and $N$.
  For this, we need to show that there exists $\varepsilon > 0$ such that $|G_{nj}| \le 1 - \varepsilon$ for all $j \neq n$. The maximum value of $|G_{nj}|$ for $n \neq j$ occurs for $(j,n) = (0,1)$.
  Moreover, by the Taylor expansion of cosine,
  \begin{equation*}
    \frac{1}{|G_{10}|} -1 = \frac{|z| (1 - \cos(2 \pi / N))}{1 - |z|}
    \simeq \frac{1}{N^2(1 - |z|)},
  \end{equation*}
  which is at least $1$ by assumption. This yields the claim about weak separation.

  It remains to show that $ \sup_j \sum_{n \neq j} |G_{n j}|^2 \le C$ for some universal constant $C$; see Proposition \ref{prop:uniformly_separated}.
  To see this, we again use Taylor expansion of cosine to find $\delta > 0$ such that $\cos(2 \pi n/N) \le 1 - \delta (n/N)^2$
  for all $1 \le n \le \lceil N/2 \rceil$.
  Thus,
  \begin{align*}
    \sup_j \sum_{n \neq j} |G_{nj}|^2
    = \sum_{n=1}^{N-1} |G_{n 0}|^2
    \lesssim &\sum_{n= 1}^{\lceil N/2 \rceil} \frac{(1 - |z|)^2}{(1 - |z| \cos(2 \pi n /N))^2} \\
    \le &\sum_{n=1}^{\lceil N/2 \rceil} \frac{(1 - |z|)^2}{(1 - |z| (1 - \delta (n/N)^2))^2} \\
    \le& (1 - |z|)^2 \sum_{n=1}^{\lceil N/2 \rceil} \frac{N^4}{\delta^2 n^4} \\
    \le &\frac{1}{\delta^2} \sum_{n=1}^\infty \frac{1}{n^4},
  \end{align*}
  where the final estimate follows from the upper bound on $N$.
\end{proof}

To deal with uniform separation, we also require the following estimate.

\begin{lemma}
  \label{lem:cos_estimate}
  There exists a universal constant $C \ge 0$ such that for all $0 \le t \le \frac{\pi}{2}$, all $z \in \mathbb{D}$
  and all $N \in \mathbb{N}$ with $N \ge \frac{1}{2(1 - |z|)^{1/2}}$,
  the following estimate holds:
  \begin{equation*}
    \frac{1}{N} \sum_{l=0}^{\lfloor \frac{N}{4} \rfloor} \frac{1}{|1 - z \cos(t - \frac{2 \pi l}{N})|^2} \le C \frac{1}{|1 - z|^{3/2}}.
  \end{equation*}
\end{lemma}

\begin{proof}
  We may assume that $N \ge 8$; otherwise, $|z|$ is bounded away from $1$, and the result is immediate.
  We write $t = \frac{2 \pi m}{N} + s$ for some $s$ with $0 \le s < \frac{2 \pi}{N}$ and $m \in \mathbb{Z}$
  with $0 \le m \le \lfloor \frac{N}{4} \rfloor$. Then
  \begin{equation*}
    \frac{1}{N} \sum_{l=0}^{\lfloor \frac{N}{4} \rfloor} \frac{1}{|1 - z \cos(t - \frac{2 \pi l}{N})|^2}
    \le \frac{1}{N} \sum_{l=- \lfloor \frac{N}{4} \rfloor}^{\lfloor \frac{N}{4} \rfloor}
    \frac{1}{|1 - z \cos(s - \frac{2 \pi l}{N})|^2}.
  \end{equation*}
  We first deal with the summand corresponding to $l = - \lfloor{\frac{N}{4}} \rfloor$, for which we
  observe that $\frac{\pi}{4} \le s - \frac{2 \pi l }{N} \le \frac{3 \pi}{4}$,
  so that summand is bounded uniformly in $z,N$ and $s$.
  For the remaining summands, we have $\cos(s - \frac{2 \pi l}{N}) \ge 0$.

  Let $I = \{ l \in \mathbb{Z}: - \lfloor \frac{N}{4} \rfloor + 1 \le l \le  \lfloor \frac{N}{4} \rfloor \}$
  be the set of remaining indices.
  We break up $I$ into
  \begin{equation*}
    I_1 = \{l \in I: |l| \le \lfloor 2 N |1 - z|^{1/2} \rfloor \}
  \end{equation*}
  and $I_2 = I \setminus I$.
  For $I_1$, we use the basic estimate
  $|1 - z| \le 2 |1 - r z|$ for all $r \in [0,1]$ and $z \in \mathbb{D}$
  and obtain
  \begin{equation*}
    \frac{1}{N} \sum_{l \in I_1}
    \frac{1}{|1 - z \cos(s - \frac{2 \pi l}{N})|^2}
    \lesssim \frac{1}{N}
    \sum_{l \in I_1}
    \frac{1}{|1  - z|^2} \lesssim \frac{1}{|1 - z|^{3/2}}.
  \end{equation*}
  For $I_2$, we use the Taylor expansion of cosine to bound
  $1 - \cos(s - \frac{2 \pi l}{N}) \gtrsim (s - \frac{2 \pi l}{N})^2$, so
  \begin{align*}
    &\frac{1}{N} \sum_{l \in I_2}
    \frac{1}{|1 - z \cos(s - \frac{2 \pi l}{N})|^2}\\
    \le & \frac{1}{N} \sum_{l \in I_2}
    \frac{1}{(1 - \cos(s - \frac{2 \pi l}{N}))^2 } \\
    \lesssim &\frac{1}{N} \sum_{l \in I_2}
    \frac{1}{(s - \frac{2 \pi l}{N})^4} \\
    \lesssim &\frac{1}{N} \sum_{l \ge \lfloor{2 N |1 - z|^{1/2}} \rfloor} \frac{N^4}{l^4} \\
    \lesssim &\frac{N^3}{( \lfloor 2 N |1 - z|^{1/2} \rfloor)^3}
    \lesssim \frac{1}{|1 - z|^{3/2}}.
  \end{align*}
  In the final estimate, we used the assumption $N \ge \frac{1}{2 (1 - |z|)^{1/2}}$,
  so $\lfloor 2 N |1 - z|^{1/2} \rfloor \ge N |1 - z|^{1/2}$.
\end{proof}

The following result makes it possible to turn counterexamples on the disc
into counterexamples on the ball.

\begin{proposition}
  \label{prop:ball_counterexamples}
  Let $(z_n)$ be a sequence in $\mathbb{D}$, let $z_n^{1/2}$ be any square root of $z_n$, let $N_n = \lfloor \frac{1}{(1 - |z_n|)^{1/2}} \rfloor$
  and define
  \begin{equation*}
    W = \Big\{ 2^{-1/2} z_n^{1/2} (e^{\frac{2 \pi i j}{N_n}}, e^{\frac{- 2 \pi i j}{N_n}}): 0 \le j \le \lfloor \frac{N_n}{4} \rfloor, n \in \mathbb{N} \Big\} \subset \mathbb{B}_2.
  \end{equation*}
  Then:
  \begin{enumerate}
    \item[(a)] If $(z_n)$ is strongly separated for $\mathcal{D}_{1/2}$, then $W$ is strongly separated for $H^2_2$.
    \item[(b)] If $(z_n)$ generates an infinite measure for $\mathcal{D}_{1/2}$, then $W$ generates an infinite measure for $H^2_2$.
    \item[(c)] If $(z_n)$ fails the Carleson condition for $\mathcal{D}_{1/2}$, then $W$ fails the Carleson condition
      for $H^2_2$.
    \item[(d)] If $(z_n)$ is uniformly separated for $\mathcal{D}_{1/2}$, then $W$ is uniformly separated for $H^2_2$.
  \end{enumerate}
\end{proposition}

\begin{proof}
  Let $w_{(n,j)} = 
    2^{-1/2} z_n^{1/2} (e^{\frac{2 \pi i j}{N_n}}, e^{\frac{- 2 \pi i j}{N_n}})$
    and observe that $r(w_{(n,j)}) = z_n$.

    (a)
  Since $(z_n)$ is strongly separated for $\mathcal{D}_{1/2}$, there exists a constant $C_1 \ge 0$
  and multipliers $\varphi_n \in \Mult(\mathcal{D}_{1/2})$ with $\|\varphi_n\|_{\Mult(\mathcal{D}_{1/2})} \le C_1$
  and $\varphi_n(z_m) = \delta_{n m}$.
  Lemma \ref{lem:US_ball} (see also the discussion following Proposition \ref{prop:uniformly_separated}) shows that for each $n$, the points $(w_{(n,j)})_{j=0}^{N_n-1}$
  are strongly separated in $H^2_2$ with strong separation constant independent of $n$.
  Thus, there exists a constant $C_2 \ge 0$ and multipliers $\psi_{(n,l)} \in \Mult(H^2_2)$
  with $\|\psi_{(n,l)}\|_{\Mult(H^2_2)} \le C_2$ such that $\psi_{(n,l)}(w_{(n,j)}) = \delta_{lj}$.
  Define $\theta_{(n,l)} = (\varphi_n \circ r) \cdot \psi_{(n,l)}$. Lemma \ref{lem:embedding} shows
  shows that $\|\theta_{(n,l)}\|_{\Mult(H^2_2)} \le C$ for some constant $C$.
  Moreover,
  \begin{equation*}
    \theta_{(n,l)}(w_{(m,j)}) = \varphi_n(z_m) \psi_{(n,l)}(w_{(m,j)})
    = \delta_{ (n,l) (m,j)}.
  \end{equation*}
  This shows that $W$ is strongly separated for $H^2_2$.

  (b) and (c)
  Let $f \in \mathcal{D}_{1/2}$. Then
    \begin{align*}
      &\sum_n |f(z_n)|^2 (1 - |z_n|^2)^{1/2}\\
      = &\sum_n \frac{(1 - |z_n|^2)^{1/2}}{\lfloor \frac{N_n}{4} \rfloor + 1} \sum_{j=0}^{\lfloor \frac{N_n}{4} \rfloor } |(f \circ r)(w_{(n,j)})|^2 \\
      \lesssim & \sum_n (1 - |z_n|^2) \sum_j |(f \circ r)(w_{(n,j)})|^2 \\
      \lesssim &\sum_{(n,j)} |(f \circ r)(w_{(n,j)})|^2 (1 - \|w_{(n,j)}\|^2).
    \end{align*}
    Choosing $f=1$, we obtain (b).

    To prove (c), recall
    that $\|f \circ r\|_{H^2_d} \lesssim \|f\|_{\mathcal{D}_{1/2}}$ by Lemma \ref{lem:embedding},
    so if $W$ satisfies the Carleson measure condition for $H^2_2$, then $(z_n)$ satisfies the Carleson
    measure condition for $\mathcal{D}_{1/2}$.

    (d) Suppose that $(z_n)$ is uniformly separated for $\mathcal{D}_{1/2}$. Part (a) shows that that $W$
    is in particular weakly separated for $H^2_2$, so we have to show that
    \begin{equation*}
      \sup_{n,j} \sum_{(m,l)} \frac{(1 - \|w_{(n,j)}\|^2) (1 - \|w_{(m,l)}\|^2)}{ |1 - \langle w_{(n,j)}, w_{(m,l)} \rangle |^2 } < \infty.
    \end{equation*}
    To this end, we estimate with the help of Lemma \ref{lem:cos_estimate}
    \begin{align*}
      &\sum_{m=0}^\infty \sum_{l=0}^{\lfloor \frac{N_m}{4} \rfloor}
      \frac{(1 - \|w_{(n,j)}\|^2) (1 - \|w_{(m,l)}\|^2)}{ |1 - \langle w_{(n,j)}, w_{(m,l)} \rangle |^2 }\\
      =& \sum_{m=0}^\infty  \sum_{l=0}^{\lfloor \frac{N_m}{4} \rfloor} \frac{(1 - |z_n|) (1 - |z_m|)}{|1 - z_n^{1/2} \overline{z_m^{1/2}} \cos( 2 \pi ( \frac{k}{N_n} - \frac{l}{N_m}))|^2} \\
      \lesssim &\sum_{m=0}^\infty \frac{N_m (1 - |z_n|)(1 - |z_m|)}{|1 - z_n^{1/2} \overline{z_m^{1/2}}|^{3/2}} \\
      \lesssim &\sum_{m=0}^\infty \frac{(1 - |z_n|)^{1/2} (1 - |z_m|)^{1/2}}{ | 1- z_n^{1/2} \overline{z_m^{1/2}}|} \\
      \lesssim &\sum_{m=0}^\infty \frac{(1 - |z_n|^2)^{1/2} (1 - |z_m|^2)^{1/2}}{ | 1- z_n \overline{z_m}|},
    \end{align*}
    where in the penultimate step, we used the definition of $N_m$ and the estimate $|1 - z_n^{1/2} \overline{z_m^{1/2}}|
    \ge \frac{1}{2} (1 - |z_n|)$, and in the last step, we applied the inequality $|1 - z_n^{1/2} \overline{z_m^{1/2}}| \ge \frac{1}{2} |1 - z_n \overline{z_m}|$. Uniform separation for $\mathcal{D}_{1/2}$ shows that the right-hand
    side above is bounded uniformly in $n$, hence $W$ is uniformly separated for $H^2_2$.
\end{proof}

As a consequence, we can establish Theorem \ref{thm:counterexamples_intro} (a) and (c) for $H^2_d$.

\begin{corollary}
  \label{cor:DA_counterexamples}
  Let $d \ge 2$. There exists a sequence in $\mathbb{B}_d$ that is strongly separated and that generates an infinite measure
  for $H^2_d$. Moreover, there exists a sequence that is uniformly separated, but not interpolating for $H^2_d$.
\end{corollary}

\begin{proof}
  If $d = 2$, then the result follows from
  Proposition \ref{prop:D_a_unif_sep_not_int}, Example \ref{exa:infinite_measure} and Proposition \ref{prop:ball_counterexamples}.
  For general $d \ge 2$, we embed $\mathbb{B}_2$ into $\mathbb{B}_d$ in the obvious way.
  Since the restriction of the kernel of $H^2_d$ to $\mathbb{B}_2$ is the kernel of $H^2_2$, this gives the desired counterexamples for $d \ge 2$.
\end{proof}

\begin{remark}
  By embedding $\mathbb{D}$ into $\mathbb{B}_d$, one also obtains a strongly separated infinite measure sequence
  and a uniformly separated not interpolating sequence for the space on $\mathbb{B}_d$ with kernel
  $\frac{1}{(1 - \langle z,w \rangle)^a }$, where $0 < a < 1$.
\end{remark}

Finally, we establish Theorem \ref{thm:counterexamples_intro} (b).

\begin{proposition}
  Let $\mathcal{H}$ be one of the spaces $H^2_d$ for $d \ge 2$, or $\mathcal{D}_a$ for $0 < a < 1$.
  Then there exists a strongly separated sequence that generates a finite measure and that is not uniformly separated.
\end{proposition}

\begin{proof}
  By Proposition \ref{prop:DA_infinite_measure} and Corollary \ref{cor:DA_counterexamples},
  there exists a sequence $(z_n)$ that is strongly separated, but that generates an infinite measure.
  Let $Z_N = \{z_n: 1 \le n \le N\}$. Let $G_N$ be the Gramian of the points in $Z_N$.
  Since $(z_n)$ is simply interpolating by Theorem \ref{thm:simply_interpolating}, there exists
  $\varepsilon > 0$ such that $G_N \ge 2 \varepsilon I$ for all $N$.
  Moreover, since $(z_n)$ generates an infinite measure, the $\ell^2$ norm of the first column of $G_N$
  tends to infinity; see Proposition \ref{prop:US_FM}.
  
  We now apply Lemma \ref{lem:assembly} with the given $\varepsilon$ and obtain automorphisms $\varphi_N$.
  Let $Z = \bigcup_{N \in \mathbb{N}} \varphi_N(Z_N)$. Then the Gramian $G$ of the points in 
  $Z$ is bounded below, hence $Z$ is simply interpolating and thus strongly separated.
  Moreover, the $\ell^2$ norms of the columns are unbounded, so uniform separation fails.
  Finally, each column individually belongs to $\ell^2$, so $Z$ generates a finite measure.
\end{proof}

\bibliographystyle{amsplain}
\bibliography{bibliography}

\providecommand{\bysame}{\leavevmode\hbox to3em{\hrulefill}\thinspace}
\providecommand{\MR}{\relax\ifhmode\unskip\space\fi MR }
\providecommand{\MRhref}[2]{%
  \href{http://www.ams.org/mathscinet-getitem?mr=#1}{#2}
}
\providecommand{\href}[2]{#2}
\begin{thebibliography}{10}

\bibitem{Agler88a}
J.~Agler, \emph{Some interpolation theorems of {N}evanlinna-{P}ick type},
  Preprint, 1988.

\bibitem{AM00}
J.~Agler and J.~E. McCarthy, \emph{Complete {N}evanlinna-{P}ick kernels}, J.
  Funct. Anal. \textbf{175} (2000), no.~1, 111--124. \MR{1774853 (2001h:47019)}

\bibitem{AM02}
\bysame, \emph{Pick interpolation and {H}ilbert function spaces}, Graduate
  Studies in Mathematics, vol.~44, American Mathematical Society, Providence,
  RI, 2002. \MR{1882259 (2003b:47001)}

\bibitem{AHM+17}
A.~Aleman, M.~Hartz, J.~E. McCarthy, and S.~Richter, \emph{Interpolating
  sequences in spaces with the complete {P}ick property}, Int. Math. Res. Not.
  IMRN (2019), no.~12, 3832--3854.

\bibitem{AHM+20a}
A.~Aleman, M.~Hartz, J.~E. McCarthy, and S.~Richter, \emph{Multiplier tests and
  subhomogeneity of multiplier algebras}, Doc. Math. \textbf{27} (2022),
  719--764. \MR{4432527}

\bibitem{AH10}
E.~Amar and A.~Hartmann, \emph{Uniform minimality, unconditionality and
  interpolation in backward shift invariant subspaces}, Ann. Inst. Fourier
  (Grenoble) \textbf{60} (2010), no.~6, 1871--1903. \MR{2791649}

\bibitem{Arcozzi16}
N.~Arcozzi, R.~Rochberg, and E.~Sawyer, \emph{Onto interpolating sequences for
  the {D}irichlet space}, arXiv:1605.02730 (2016).

\bibitem{ARS+11}
N.~Arcozzi, R.~Rochberg, E.~Sawyer, and B.~D. Wick, \emph{Distance functions
  for reproducing kernel {H}ilbert spaces}, Function spaces in modern analysis,
  Contemp. Math., vol. 547, Amer. Math. Soc., Providence, RI, 2011, pp.~25--53.
  \MR{2856478 (2012k:46036)}

\bibitem{ArcRochSawWick19}
Nicola Arcozzi, Richard Rochberg, Eric~T. Sawyer, and Brett~D. Wick, \emph{The
  {D}irichlet space and related function spaces}, Mathematical Surveys and
  Monographs, vol. 239, American Mathematical Society, Providence, RI, 2019.
  \MR{3969961}

\bibitem{Arveson98}
W.~Arveson, \emph{Subalgebras of {$C\sp *$}-algebras. {III}. {M}ultivariable
  operator theory}, Acta Math. \textbf{181} (1998), no.~2, 159--228.
  \MR{1668582 (2000e:47013)}

\bibitem{Beardon71}
A.~F. Beardon, \emph{Inequalities for certain {F}uchsian groups}, Acta Math.
  \textbf{127} (1971), 221--258. \MR{286996}

\bibitem{Berndtsson85}
B.~Berndtsson, \emph{Interpolating sequences for {$H^\infty$} in the ball},
  Nederl. Akad. Wetensch. Indag. Math. \textbf{47} (1985), no.~1, 1--10.
  \MR{783001}

\bibitem{Bishop94}
C.~J. Bishop, \emph{Interpolating sequences for the {D}irichlet space and its
  multipliers}, Preprint,
  http://www.math.stonybrook.edu/~bishop/papers/mult.pdf (1994).

\bibitem{Bogdan96}
K.~Bogdan, \emph{On the zeros of functions with finite {D}irichlet integral},
  Kodai Math. J. \textbf{19} (1996), no.~1, 7--16. \MR{1374458}

\bibitem{BCM+15}
M.~Bownik, P.~Casazza, A.~W. Marcus, and D.~Speegle, \emph{Improved bounds in
  {W}eaver and {F}eichtinger conjectures}, J. Reine Angew. Math. \textbf{749}
  (2019), 267--293. \MR{3935905}

\bibitem{Carleson58}
L.~Carleson, \emph{An interpolation problem for bounded analytic functions},
  Amer. J. Math. \textbf{80} (1958), 921--930. \MR{0117349 (22 \#8129)}

\bibitem{CCL+05}
P.~G. Casazza, O.~Christensen, A.~M. Lindner, and R.~Vershynin, \emph{Frames
  and the {F}eichtinger conjecture}, Proc. Amer. Math. Soc. \textbf{133}
  (2005), no.~4, 1025--1033. \MR{2117203}

\bibitem{Chalmoukis21}
N.~Chalmoukis, \emph{Onto interpolation for the {D}irichlet space and for
  ${H}_1(\mathbb{D})$}, Adv. Math. \textbf{381} (2021), 107634.

\bibitem{Chalmoukis23}
\bysame, \emph{A note on simply interpolating sequences for the {D}irichlet
  space}, New York J. Math. \textbf{29} (2023), 193--202.

\bibitem{Daase83}
D.~Daase, \emph{Automorphe {F}ormen und {C}arleson-{M}engen}, Complex Variables
  Theory Appl. \textbf{2} (1983), no.~1, 51--65. \MR{707785}

\bibitem{DalBo11}
F.~Dal'Bo, \emph{Geodesic and horocyclic trajectories}, Universitext,
  Springer-Verlag London, Ltd., London; EDP Sciences, Les Ulis, 2011,
  Translated from the 2007 French original. \MR{2766419 (2011j:37056)}

\bibitem{Garnett07}
J.~B. Garnett, \emph{Bounded analytic functions}, first ed., Graduate Texts in
  Mathematics, vol. 236, Springer, New York, 2007. \MR{2261424 (2007e:30049)}

\bibitem{Hartz17a}
M.~Hartz, \emph{On the isomorphism problem for multiplier algebras of
  {N}evanlinna-{P}ick spaces}, Canad. J. Math. \textbf{69} (2017), no.~1,
  54--106. \MR{3589854}

\bibitem{Hartz22a}
\bysame, \emph{An invitation to the {D}rury-{A}rveson space}, arXiv:2204.01559
  (2022).

\bibitem{Hartz20}
\bysame, \emph{Every complete {P}ick space satisfies the column-row property},
  Acta Math. (to appear), arXiv:2005.09614.

\bibitem{Katok92}
S.~Katok, \emph{Fuchsian groups}, Chicago Lectures in Mathematics, University
  of Chicago Press, Chicago, IL, 1992. \MR{1177168 (93d:20088)}

\bibitem{circulant}
I.~Kra and S.~R. Simanca, \emph{On circulant matrices}, Notices of the American
  Math. Soc. \textbf{59} (2012), no.~3, 368-- 377.

\bibitem{KS}
Mark Krosky and Alexander~P. Schuster, \emph{Multiple interpolation and
  extremal functions in the {B}ergman spaces}, J. Anal. Math. \textbf{85}
  (2001), 141--156. \MR{1869605}

\bibitem{MSS15}
A.~W. Marcus, D.~A. Spielman, and N.~Srivastava, \emph{Interlacing families
  {II}: {M}ixed characteristic polynomials and the {K}adison-{S}inger problem},
  Ann. of Math. (2) \textbf{182} (2015), no.~1, 327--350. \MR{3374963}

\bibitem{MS94a}
D.~Marshall and C.~Sundberg, \emph{Interpolating sequences for the multipliers
  of the {Dirichlet} space}, Preprint,
  \url{https://citeseerx.ist.psu.edu/viewdoc/download?doi=10.1.1.559.4031&rep=rep1&type=pdf},
  1994.

\bibitem{McCullough92}
S.~McCullough, \emph{Carath\'eodory interpolation kernels}, Integral Equations
  Operator Theory \textbf{15} (1992), no.~1, 43--71. \MR{1134687}

\bibitem{Pommerenke76}
Ch. Pommerenke, \emph{On automorphic forms and {C}arleson sets}, Michigan Math.
  J. \textbf{23} (1976), no.~2, 129--136. \MR{409800}

\bibitem{Quiggin93}
P.~Quiggin, \emph{For which reproducing kernel {H}ilbert spaces is {P}ick's
  theorem true?}, Integral Equations Operator Theory \textbf{16} (1993), no.~2,
  244--266. \MR{1205001 (94a:47026)}

\bibitem{K.V.RAJESWARA69}
K.~V.~Rajeswara Rao, \emph{Fuchsian groups of convergence type and {P}oincare
  series of dimension -2}, Journal of Mathematics and Mechanics \textbf{18}
  (1969), no.~7, 629--644.

\bibitem{Rudin1980}
W.~Rudin, \emph{Function theory in the unit ball of $\mathbb{C}^n$}, Springer
  New York, 1980.

\bibitem{SS00}
A.~P. Schuster and K.~Seip, \emph{Weak conditions for interpolation in
  holomorphic spaces}, Publ. Mat. \textbf{44} (2000), no.~1, 277--293.
  \MR{1775765}

\bibitem{SS98}
A.~P. Schuster and Kristian Seip, \emph{A {C}arleson-type condition for
  interpolation in {B}ergman spaces}, J. Reine Angew. Math. \textbf{497}
  (1998), 223--233. \MR{1617432}

\bibitem{Shalit13}
O.~Shalit, \emph{Operator theory and function theory in {D}rury-{A}rveson space
  and its quotients}, Operator Theory (Daniel Alpay, ed.), Springer, 2015,
  pp.~1125--1180.

\bibitem{SS61}
H.~S. Shapiro and A.~L. Shields, \emph{On some interpolation problems for
  analytic functions}, Amer. J. Math. \textbf{83} (1961), 513--532. \MR{0133446
  (24 \#A3280)}

\bibitem{Shirokov90}
N.~A. Shirokov, \emph{Fuchsian groups with {C}arleson limit sets}, Funktsional.
  Anal. i Prilozhen. \textbf{24} (1990), no.~4, 92--93. \MR{1092814}

\bibitem{TW71}
B.~A. Taylor and D.~L. Williams, \emph{Zeros of {L}ipschitz functions analytic
  in the unit disc}, Michigan Math. J. \textbf{18} (1971), 129--139.
  \MR{283176}

\bibitem{Vasin03}
A.~V. Vasin, \emph{The limit set of a {F}uchsian group and the {D}ynkin lemma},
  Zap. Nauchn. Sem. S.-Peterburg. Otdel. Mat. Inst. Steklov. (POMI)
  \textbf{303} (2003), no.~Issled. po Line\u{\i}n. Oper. i Teor. Funkts. 31,
  89--101, 322. \MR{2037532}

\bibitem{Zhu05}
K.~Zhu, \emph{Spaces of holomorphic functions in the unit ball}, Graduate Texts
  in Mathematics, vol. 226, Springer-Verlag, New York, 2005. \MR{2115155
  (2006d:46035)}

\end{thebibliography}

\end{document}